\documentclass[12pt,reqno]{amsart}
\usepackage{amsmath,amsthm,amssymb}
\usepackage{caption}
\usepackage{mathtools}
\usepackage{cite}
\usepackage{enumitem}
\usepackage{color}
\usepackage{url}
\usepackage[dvipsnames]{xcolor}
\usepackage{graphicx}
\usepackage{xspace}
\usepackage[unicode=true,colorlinks]{hyperref}
\hypersetup{
	linkcolor=blue,
	citecolor=blue,
}
\usepackage{cleveref}
\usepackage{leftidx}
\usepackage[margin=1.2in]{geometry}
\usepackage{comment}
\usepackage{tikz}

\usepackage{cleveref}

\newtheorem{theorem}{Theorem}[section]
\newtheorem{lemma}[theorem]{Lemma}
\newtheorem{definition}[theorem]{Definition}
\newtheorem{proposition}[theorem]{Proposition}
\theoremstyle{remark}
\newtheorem{remark}[theorem]{Remark}
\newtheorem{notation}[theorem]{Notation}

\newcommand{\Z}{\mathbb{Z}}
\newcommand{\R}{\mathbb{R}}

\DeclareMathOperator{\hol}{hol}
\DeclareMathOperator{\SL}{SL}
\DeclareMathOperator{\GL}{GL}
\DeclareMathOperator{\covol}{covol}
\DeclareMathOperator{\dist}{dist}

\newcommand{\bfe}{\mathbf{e}}

\DeclarePairedDelimiter\abs{\lvert}{\rvert}%
\DeclarePairedDelimiter\norm{\lVert}{\rVert}%

\makeatletter
\let\oldabs\abs
\def\abs{\@ifstar{\oldabs}{\oldabs*}}
\let\oldnorm\norm
\def\norm{\@ifstar{\oldnorm}{\oldnorm*}}
\makeatother

\title[Non-uniquely ergodic directions in eigenform loci]{Hausdorff dimension of non-uniquely ergodic directions in eigenform loci}
\author{Yuming Wei}
\author{Pengyu Yang}
\date{April 2026}

\begin{document}
\begin{abstract}
  We consider the eigenform locus $\mathcal{E}_D$ in $\mathcal{H}(1,1)$ where $D$ is not a square. We prove that for any translation surface $(X,\omega) \in \mathcal{E}_D$, the set of non-uniquely ergodic directions has Hausdorff dimension $1/2$, except when $D=5$ and $(X,\omega)$ lies in the Teichm\"uller curve generated by the regular decagon.
\end{abstract}
\maketitle

\section{Introduction}
A \textit{translation surface} is a pair $(X, \omega)$, where $X$ is a compact Riemann surface and $\omega$ is a non-zero holomorphic $1$-form on $X$. The moduli space of genus $g$ translation surfaces is stratified according to the multiplicities of the zeros of $\omega$, and there is a natural $\GL_2^+(\R)$-action on each stratum. For $g=2$, the stratum $\mathcal{H}(1,1)$ comprises surfaces with two simple zeros. Within this stratum, we consider the \textit{eigenform locus} $\mathcal{E}_D$, which is a closed, $\mathrm{GL}_2^+(\mathbb{R})$-invariant subvariety. 

\begin{definition}
We say that a pair $(X,\omega)$ is an eigenform if $\operatorname{Jac}(X)$ has real multiplication by a real quadratic order $\mathcal{O}_D$ and $\omega$ is an eigenform for this action. 
\end{definition}

For a translation surface $(X, \omega)$, any direction $\theta$ defines a \textit{straight-line flow} (or directional flow) $\{f_t^\theta\}_{t \in \mathbb{R}}$ on $X$. For every point $x \in X$, the \textit{orbit} of $x$ in direction $\theta$ is the set 
\[ \mathcal{O}_\theta(x) = \{ f_t^\theta(x) : t \in \mathbb{R} \}. \]

A direction $\theta$ is said to be \textit{uniquely ergodic} if for every continuous function $\varphi \in C(X)$, and for every $x \in X$ whose orbit does not meet the singularities of $\omega$, the time average of $\varphi$ converges to the space average:
\begin{equation}
\lim_{T \to \infty} \frac{1}{T} \int_0^T \varphi(f_t^\theta(x)) \, dt = \frac{1}{\text{Area}(X)} \int_X \varphi \, d\mu_{area}.
\end{equation}

We then define the \textit{non-uniquely ergodic directions} (NUE) as the set:
\begin{equation}
\Theta_{NUE}(X, \omega) = \{ \theta : f_t^\theta \text{ is not uniquely ergodic} \}.
\end{equation}

The purpose of this paper is to prove the following theorem.
\begin{theorem}\label{thm:main theorem}
Let $D>0$ be a nonsquare discriminant of a real quadratic order. Let $\mathcal{E}_D$ be the eigenform locus in the stratum $\mathcal{H}(1,1)$. Then, except when $D=5$ and $(X,\omega)$ lies on the $\mathrm{GL}_2^+(\R)$-orbit through the regular decagon, every translation surface $(X, \omega) \in \mathcal{E}_D$ satisfies the Hausdorff dimension of the set of non-uniquely ergodic directions 
    \begin{align}
        \dim_H(\Theta_{NUE}(X, \omega)) = \frac{1}{2}.
    \end{align}
\end{theorem}

Our proof of the Hausdorff-dimension proceeds by constructing
a tree of splittings whose limit set is contained in a subset of non-uniquely ergodic directions and has Hausdorff dimension $\frac{1}{2}$.
The central notion in the construction is that of an $H$-normal splitting. We prove that every sufficiently long $H$-normal splitting
has sufficiently many children that are again $H$-normal; see
Theorem~\ref{th:induction step}.

To obtain these children, we first use a primitive lattice-point count to produce many candidate twisting vectors. A candidate $v$ is called $H$-good if, for at least one parameter $s\in\{1,2,3\}$, the
resulting child splitting is $H$-normal. The use of three possible twist parameters is a crucial feature of the argument. Three is the smallest number required by our method, since the estimates associated with three consecutive twists allow us to control the bad candidates; see Lemma~\ref{lem:bad-child-estimate}. For each $H$-good twist $v$, we choose one such parameter $s(v)$, and distinct $H$-good twists yield distinct $H$-normal children. Combining the primitive lattice-point
count with the bad-child estimate therefore gives the branching estimate in Theorem~\ref{th:induction step}.

\subsection{History}
A translation surface is flat everywhere except at a finite set of singular points, which correspond to the zeros of the holomorphic $1$-form $\omega$. Geometrically, these singularities are conical points with cone angles of the form $\theta = 2\pi(k+1)$, where $k \in \mathbb{N}$ is the multiplicity (or order) of the zero. 

The Hodge bundle over the moduli space $\mathcal{M}_g$ is partitioned into \textit{strata} according to the multiplicities of these zeros. A specific stratum is denoted by $\mathcal{H}(k_1, \dots, k_n)$, consisting of pairs $(X, \omega)$ where the $1$-form $\omega$ has $n$ distinct zeros with multiplicities $k_1, \dots, k_n$. The sum of these multiplicities must satisfy:
\begin{equation*}
    \sum_{i=1}^n k_i = 2g - 2.
\end{equation*}
Each stratum is equipped with a natural Masur-Veech measure $\mu_{MV}$, induced by the pullback of Euclidean measure through local period coordinate charts. 

The study of ergodic directions on translation surfaces has a rich history in Teichmüller dynamics. The Veech dichotomy establishes a rigid classification for Veech surfaces: every direction $\theta$ is either completely periodic or uniquely ergodic. However, Veech surfaces constitute a highly restricted, measure-zero subset of the moduli space. In 1986, Kerckhoff-Masur-Smillie\cite{KMS} showed that, for any translation surface, almost every direction is uniquely ergodic. Equivalently, the Lebesgue measure of $\Theta_{NUE}(X, \omega)$ is 0. 
A significant refinement was provided by Masur and Smillie \cite{MS}, who demonstrated that for every component $\mathcal{C}$ of each stratum $\mathcal{H}(\alpha)$ of genus $g \ge 2$, there exists a constant $c(\mathcal{C})$ such that for $\mu_{MV}$-almost every translation surface $(X, \omega) \in \mathcal{C}$, the Hausdorff dimension of the set of non-uniquely ergodic directions equals $c(\mathcal{C})$.
In 1992, Masur \cite{Ma2} gave an upper bound $c(\mathcal{C})\le \frac{1}{2}$.
Later Athreya and Chaika \cite{AC} showed that $c(\mathcal{C})=\frac{1}{2}$ for $\mathcal{C}=\mathcal{H}(2)$. More recently, Chaika and Masur\cite{CM} showed that $c(\mathcal{C})=\frac{1}{2}$ for $\mathcal{C}=\mathcal{H^{\mathrm{hyp}}}(2g-2)$ and $\mathcal{C}=\mathcal{H^{\mathrm{hyp}}}(g-1,g-1)$.

When $g=2$, although the results of Athreya-Chaika \cite{AC} and Chaika-Masur\cite{CM} establish a uniform dimension for $\mu_{MV}$-almost every surface within a stratum, they do not directly apply to surfaces belonging to the eigenform locus $\mathcal{E}_D$. 
In both $\mathcal{H}(2)$ and $\mathcal{H}(1,1)$, the locus $\mathcal{E}_D$ constitutes a closed invariant subvariety of measure zero. 
In the stratum $\mathcal{H}(2)$, a theorem of McMullen \cite[Theorem 10.1]{McM} states that the $\GL_2^+(\R)$-orbit of every $(X,\omega)\in \mathcal{E}_D$ is a Teichm\"uller curve.
In contrast, the dynamics in the stratum $\mathcal{H}(1,1)$ are more complex. For the eigenform locus $\mathcal{E}_D\subset \mathcal{H}(1,1)$, when $D$ is square, the $\operatorname{GL}_2^+(\R)$-orbit of $X$ is a Teichm\"uller curve if after a suitable $\GL_2^+(\R)$-normalization, all absolute and relative periods lie in a common rank-two lattice(we call this translation surface a square-tiled surface); When $D$ is nonsquare, the only primitive Teichm\"uller curve generated by a form with simple zeros is the one generated by the regular decagon for $D=5$. (See \cite{McM2}.)

The systematic study of non-uniquely ergodic directions specifically within the eigenform locus for square $D$ began with the work of \text{Cheung \cite{Ch1}} in the case $D=4$. A translation surface $(X, \omega)$ in the locus $\mathcal{E}_4 \subset \mathcal{H}(1,1)$ can be realized geometrically as a branched double cover of a flat torus, typically constructed by gluing two identical flat tori along a slit $w$. In the specific case where the slit $w$ is horizontally oriented, Cheung \cite{Ch1} gave a sufficient condition ensuring that the set of non-uniquely ergodic directions has Hausdorff dimension $1/2$. In 2011, Cheung, Hubert and Masur \cite{2011Dichotomy} showed that the Hausdorff dimension of the set of nonergodic directions takes only the values $0$ and $1/2$, a value explicitly determined by the \textit{continued fraction expansion} of the length of $w$. Recently, the first author \cite{Wei} generalized this result by removing the restriction on the orientation of $w$, establishing the dimensional estimates for slits in arbitrary directions.  

We focus primarily on the case where $D$ is nonsquare. We show that the dimension property becomes intrinsic to the locus: for every translation surface in the eigenform locus which does not lie on the $\operatorname{GL}_2^+(\mathbb{R})$-orbit through the regular decagon, the Hausdorff dimension of non-uniquely ergodic directions is $1/2$.

\section{Splittings and Dehn twists}

Given a translation surface $(X,\omega)\in \mathcal{H}(1,1)$, a splitting $(L_1,L_2,w)$ is a decomposition of $(X,\omega)$ as a connected sum of two flat tori $T_1, T_2$ along a slit $w\in \mathbb{C}$ joining two zeros. 
Equivalently, we can write $(X,\omega)$ as a connected sum
\begin{align}
    (X,\omega)=(T_1,\omega_1) \underset{w}{\#}\left(T_2, \omega_2\right)
\end{align}
where $(T_i,\omega_i)=(\mathbb{C}/L_i,dz)$.

\begin{figure}[htb]
\centering
\begin{tikzpicture}[scale=0.9]

    \coordinate (A) at (0,2);
    \coordinate (B) at (3,2);
    \coordinate (C) at (3,4);
    \coordinate (D) at (0,4);
    
    \coordinate (E) at (6,2.8);
    \coordinate (F) at (6,4.8);
    
    \coordinate (G) at (-0.4,0.4);
    \coordinate (H) at (2.6,0.4);
    
    \coordinate (I) at (3.5,0.4);
    \coordinate (J) at (6.5,1.2);

    \draw (D) -- (C) -- (F) -- (E) -- (B) -- (A) -- cycle;
    \draw (B) -- (C); 

    \draw (A) -- (G) -- (H) -- (B);
    
    \draw (B) -- (I) -- (J) -- (E);

    \node at (1.5,3.1) {\Large $R_1$};
    \node at (4.4,3.5) {\Large $R_2$};
    \node at (1.2,1.1) {\Large $C_1$};
    \node at (4.8,1.4) {\Large $C_2$};

\end{tikzpicture}
\caption{Normal form for $(X,\omega)$}
\label{fig:X}
\end{figure}
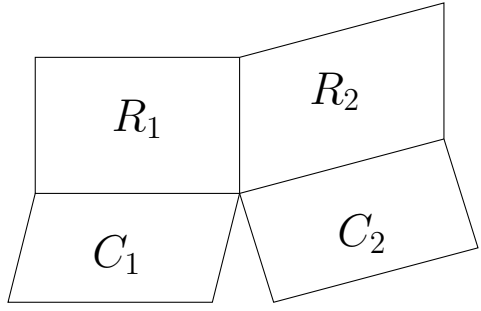

For geometric intuition, we represent $(X, \omega)$ as a polygon with paired sides identified. 
We choose a pair of primitive vectors $(v_1, v_2) \in L_1 \times L_2$ satisfying 
\begin{equation}\label{eq:orientation}
    (v_1 \times w)(v_2 \times w) > 0,
\end{equation}
and
\begin{equation}\label{eq:cylinder}
   \norm{v_i \times w} \leq \text{area}(T_i).  
\end{equation}
Then $(v_1, v_2)$ uniquely determines such a curve $\gamma$. Condition (\ref{eq:cylinder}) ensures that $v_i$ corresponds to a simple closed curve $\gamma_i$ on $T_i$ bounding a cylinder $C_i$ disjoint from the slit. The complement $X\setminus \bigcup C_i$ is an open annulus whose core curve is precisely $\gamma$. 
Then each $T_i$ decomposes into a cylinder $C_i$ and a parallelogram $R_i$ spanned by $\{v_i, w\}$. Using the $\SL_2(\mathbb{R})$-action, we normalize $R_1$ to a rectangle. We then adjoin $R_2$ to the right edge of $R_1$ (the slit $\alpha_+$) and attach $C_i$ to the lower edges of $R_i$ without overlap. The boundary identifications are as follows: the right edge of $R_2$ is identified with the left edge of $R_1$ to form the slit $\alpha_-$, while $R_i$ and $C_i$ are glued along their respective horizontal edges.

In \cite{MS}, the nonergodic directions on a translation surface $(X,\omega)$ can be constructed by producing a sequence of splittings in the following way:
\begin{theorem}[\cite{MS} Theorem 2.1]\label{thm:MS}
Let $E_n$ be a sequence of metric cylinders in $(X,\omega)$ so that $E_n$ is disjoint from $E_{n+1}$.  Then we define a sequence of partitions $P_n=\left\{A_n, B_n\right\}$ as follows. If $n$ is even then $A_n=E_n$ and $B_n=E_n^c$. If $n$ is odd then $A_n=E_n^c$ and $B_n=E_n$. We can rotate the coordinate system so that $\theta_{\infty}$ is the vertical direction. Let $h_n$ be the sum of the horizontal components of the boundary curves separating $A_n$ and $B_n$. If \\
i) $\lim _{n \rightarrow \infty} h_n=0$,\\
ii) $0<c \leq \mu\left(A_n\right) \leq c^{\prime}<1$ for some $c, c^{\prime}$,\\
iii) $\sum_{n=1}^{\infty} \mu\left(A_n \Delta A_{n+1}\right)<\infty$.

Then the vertical foliation is nonergodic.
\end{theorem}

\subsection{Dehn twist}
A Dehn twist is a specific type of self-homeomorphism of a surface X. It represents a ``cut-twist-glue'' operation along a simple closed curve $\gamma \subset X$.

Following \cite[Section 3]{CM}, we show that any splitting $(L_1, L_2, w)$ gives rise, via Dehn twists, to sufficiently many splittings $(L_1^{\prime}, L_2^{\prime}, w^{\prime})$.
In Figure~\ref{fig:saddle_conn},  we view  $C_i \subset T_i$ as a cylinder disjoint from the slit $w_{\pm}$, with core curve $\gamma_i$ having holonomy $v_i$ not parallel to $w$. Concatenating the unique translates of $\gamma_i$ passing through the midpoint of $w$ yields a simple closed curve $\gamma$ with holonomy $v_1 + v_2$.

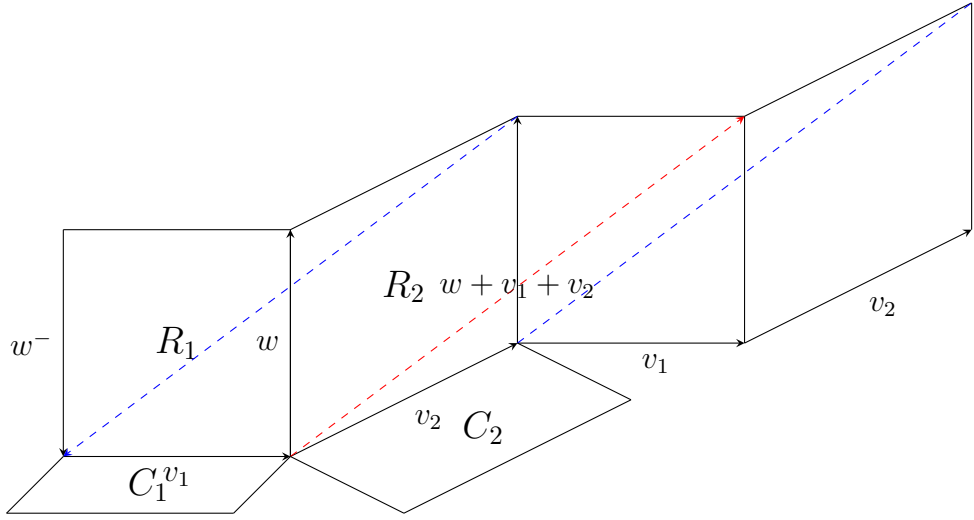
\begin{figure}[htb]
    \centering
    \begin{tikzpicture}[scale=1.5, >=stealth]
        \coordinate (A) at (0,0);
        \coordinate (B) at (2,0);      
        \coordinate (C) at (4,1);      
        \coordinate (Bw) at (2,2);
        \coordinate (D) at (6,1);
        \coordinate (E) at (8,2);
        \node at (1,1) {\large $R_1$};
        \node at (3,1.5) {\large $R_2$};
        
        \coordinate (Aw) at (0,2);     
        
        \draw[->] (A) -- node[below] {$v_1$} (B);
        \draw[->] (B) -- node[below right] {$v_2$} (C);
        \draw[->] (C) -- node[below right] {$v_1$} (D);
        \draw[->] (D) -- node[below right] {$v_2$} (E);

        \draw[->] (Aw) -- node[left] {$w^-$} (A);
        \draw[->] (B) -- node[left] {$w$} (Bw);
        \draw (B) -- ++(0,2);
        \draw (C) -- ++(0,2);
        \draw (D) -- ++(0,2);
        \draw (E) -- ++(0,2);
        \draw[->] (C) -- ++(0,2); 
        
        \draw (Aw) -- ++(2,0) -- ++(2,1) -- ++(2,0) -- ++ (2,1);
        
        \draw[dashed][->][red] (B) --node[black]{$w+v_1+v_2$} (6,3);
        \draw[dashed][->][blue] (4,3) -- (0,0);
        \draw[dashed][blue] (C) -- (8,4);

        \coordinate (F) at (-0.5,-0.5);
        \coordinate (G) at (1.5,-0.5);
        \draw (A) -- (F);
        \draw (F) -- (G);
        \draw (B) -- (G);
        \node at (0.75,-0.25) {\large $C_1$};

        \coordinate (H) at (3,-0.5);
        \coordinate (I) at (5,0.5);
        \draw (B) -- (H);
        \draw (H) -- (I);
        \draw (C) -- (I);
        \node at (3.7,0.25) {\large $C_2$};
        
    \end{tikzpicture}
    \caption{$\beta_+^{1}$ has holonomy vector $w+v_1+v_2$.}  
    \label{fig:saddle_conn}
\end{figure}

Let $\beta_{\pm}^s$ be the curves obtained by twisting $w_{\pm}$ relative to their endpoints $s$ times about $\gamma$ in the positive sense, i.e. right-twist if $s>0$, and left-twist if $s<0$. $s$ is called the shear parameter or twist parameter. In general, the geodesic representative of a twisted curve is a finite sequence of saddle connections. Let $w^s$ be the common holonomy vector of $\beta_{ \pm}^s$. Note that if $v_1 \times w>0$, i.e. $\gamma$ is positively oriented with respect to $\alpha_{ \pm}$, then 
\begin{align}\label{def: w_s}
    w_s=w=s (v_1+v_2).
\end{align}

\begin{lemma}\cite[Lemma 3.1]{CheM}\label{lem:CM}
The twisted curves $\beta_{+}^k$ and $\beta_{-}^k$ are simultaneously realized by a (single) saddle connection if and only if $w$ and $w_k$ lie on the same side of $v_1$ and $v_2$, i.e. 
\begin{align}\label{eq:positively oriented}
(v_1 \times w_k)(v_2 \times w_k)>0.
\end{align}

\end{lemma}

\begin{definition}\label{def:(L_1^s,L_2^s)}
Let $(L_1,L_2,w)$ be a splitting of $(X,\omega)$. Let $v_1\in L_1$ and $v_2\in L_2$ be primitive vectors satisfying (\ref{eq:orientation}) and (\ref{eq:cylinder}). 

Choose $u_1\in L_1$ and $u_2\in L_2$ such that $L_1=\Z v_1\oplus \Z u_1$ and $L_2=\Z v_2\oplus \Z u_2$.
For $s\in\Z$, let $w_s$ be defined as in \eqref{def: w_s}. Suppose that the triple $(w_s,v_1,v_2)$ satisfies the condition (\ref{eq:positively oriented}) in Lemma~\ref{lem:CM}. We define
\begin{align}
    L_1^s = \Z v_1 \oplus \Z (u_1+ s v_2),
\qquad
     L_2^s = \Z v_2 \oplus \Z (u_2+ s v_1).
\end{align}
The triple $(L_1^s,L_2^s,w_s)$ is called the splitting obtained
from $(L_1,L_2,w)$ by the $s$-fold Dehn twist determined by
$v_1$ and $v_2$.
\end{definition}

\section{Splittings of eigenforms}\label{sec:reduction}
In this section, we recall some results on splittings of eigenforms.
Every geometric splitting determines an algebraic splitting of the absolute homology. Applying \cite[Theorem 8.3]{McM} to this algebraic splitting gives the following characterization.
\begin{theorem}\cite{McM}\label{thm: eigenform locus}
A connected sum $(X, \omega) \cong\left(T_1, \omega_1\right)\underset{w}{\#}\left(T_2, \omega_2\right)$ lies in $\mathcal{E}_D$ if and only if there are integers satisfying $e^2+4 d=D$, and an isogeny $p: T_1 \rightarrow T_2$, such that:\\
(i) $\operatorname{deg}(p)=d \geq 1, \operatorname{gcd}(p, e)=1$\footnote{Here $\operatorname{gcd}(p, e)=\operatorname{gcd}\left(p_{11}, p_{12}, p_{21}, p_{22}, e\right)$, where $\left(p_{i j}\right) \in \mathrm{M}_2(\mathbb{Z})$ is a matrix for the map $p_*: H_1\left(T_1, \mathbb{Z}\right) \rightarrow H_1\left(T_2, \mathbb{Z}\right)$.}, and\\
(ii) $p^*\left(\omega_2\right)=\lambda \omega_1$, where $\lambda^2=e \lambda+d$.

In particular, $D$ determines $\operatorname{deg}(p)$ and $\lambda$ up to finitely many choices.
\end{theorem}

Suppose $D$ is a nonsquare. We take $e=1$ if $D\equiv 1\mod 4$, and $e=0$ if $D\equiv0\mod 4$. Hence $d=\frac{D-e^2}{4}$ is determined. We fix $\lambda$ to be the positive root of the equation $x^2-ex-d=0$ with the above choices of $e$ and $d$.

We recall the following result from \cite[Proposition 7.1]{BSW}.
\begin{proposition} \label{prop:BSW}
    A genus 2 translation surface $M$ is an eigenform for real multiplication by $\mathcal{O}_D$ if and only if there is a proper monomorphism $\rho_0:\mathcal{O}_D\to\operatorname{End}^0(H_1(M;\Z))$ such that
    \begin{equation} \label{eq:eigenform criterion}
        \hol(M,\rho_0(\lambda)\cdot\gamma) = \iota(\lambda)\hol(M,\gamma)
    \end{equation}
    for each $\lambda\in\mathcal{O}_D$ and $\gamma\in H_1(M;\Z)$.
\end{proposition}

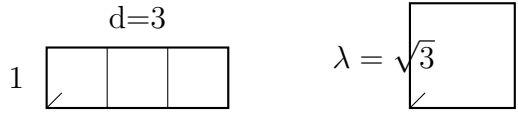
\begin{figure}[htbp]
\centering
\begin{tikzpicture}[scale=0.8]
    \def\d{3}
    \draw[thick] (0,0) rectangle (\d,1);
    \foreach \x in {1,...,\d}
        \draw (\x,0) -- (\x,1);
    \node at (-0.5, 0.5) {1};
    \node at (\d/2, 1.5)  {d=3};
    \draw (0,0) -> (0.25,0.25);
    
    \def\mu{1.732}
    \draw[thick] (\d+3, 0) rectangle (\d+3+\mu, \mu);
    \node at (\d+3+\mu/2-1.3, \mu/2) {$\lambda=\sqrt{3}$};
    \draw (\d+3,0)  -> (\d+3+0.25, 0.25);
\end{tikzpicture}
\caption{A translation surface lies in $\mathcal{E}_{12}$.}
\end{figure}

Take the following two rectangles: one with width $d$ and height $1$, the other with width and height both equal to $\lambda$. Glue the opposite sides to get two flat tori, and glue them along a slit $(x,y)$ with $0<\norm{(x,y)}_\infty<\min\{1,\lambda\}$. The resulting flat surface is denoted by $M_{D,x,y}$. Then Proposition~\ref{prop:BSW} implies that $M_{D,x,y}$ lies in $\mathcal{E}_D$. Indeed, there exists a symplectic basis $\{ \alpha_1,\beta_1,\alpha_2,\beta_2 \}$ for $H_1(M;\Z)$ such that $\alpha_1\cdot\beta_1=1, \alpha_2\cdot\beta_2=1$ and $\alpha_i\cdot\alpha_j=\beta_i\cdot\beta_j=0$ for every $1\leq i,j\leq2$; and we have $\hol(\alpha_1)=d\bfe_1$, $\hol(\beta_1)=\bfe_2$, $\hol(\alpha_2)=\lambda\bfe_1$, $\hol(\beta_2)=\lambda\bfe_2$. Let $\rho_0$ be defined by $\rho_0(\lambda)\cdot\alpha_1 = d\alpha_2$, $\rho_0(\lambda)\cdot\alpha_2 = e\alpha_2+\alpha_1$, $\rho_0(\lambda)\cdot\beta_1 = \beta_2$, $\rho_0(\lambda)\cdot\beta_2=e\beta_2+d\beta_1$. One can check that \eqref{eq:eigenform criterion} is satisfied and $\rho_0$ is proper, and thus $M_{D,x,y}$ lies in $\mathcal{E}_D$.

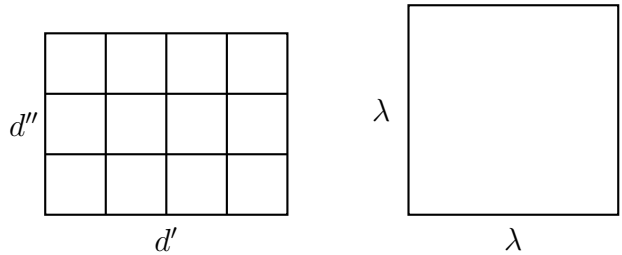
\begin{figure}[htb!]
\centering
\begin{tikzpicture}[scale=0.8, line width=0.8pt]


\draw (0,0) rectangle (4,3);

\foreach \x in {1,2,3} {

    \draw (\x,0) -- (\x,3);

}

\foreach \y in {1,2} {

    \draw (0,\y) -- (4,\y);

}


\draw (6,0) rectangle ++(3.464,3.464);


\node at (2,-0.4) {$d'$};

\node at (-0.35,1.5) {$d''$};

\node at (7.732,-0.4) {$\lambda$};

\node at (5.55,1.732) {$\lambda$};

\end{tikzpicture}
\caption{$(d'd'',\lambda)-type$}
\end{figure}

\begin{definition}[$(d,\lambda)$-type]
Let $(L_1,L_2,w)$ be a splitting of $(X,\omega)$.
We say that this splitting is of type $(d,\lambda)$ if the following two conditions hold.
\begin{itemize}
\item[(i)]There exists a factorization $d=d'd''$ 
such that
$$
L_1=\mathbb Z v_1\oplus \mathbb Z u_1,\qquad
L_2=\mathbb Z v_2\oplus \mathbb Z u_2,
$$
and
$$
\lambda v_1 = d'v_2,\qquad
{\lambda}u_1= d'' u_2+ c_v v_2
$$
for some integer $c_v$. 
\item[(ii)] Let 
\begin{align*}
    A=
\begin{pmatrix}
d'&c_v\\
0&d''
\end{pmatrix},
\end{align*}
then $\gcd(A,e)=\gcd(d',c_v,d'',e)=1$.
\end{itemize}
\end{definition}

\begin{lemma}\label{lem:v_2}
Let $(X,\omega)\in\mathcal E_D$ and $(L_1,L_2,w)$ be a splitting of type $(d,\lambda)$. For any primitive vector $v \in L_1$, there exists a divisor $\hat d(v)$ of $d$ such that $\frac{\lambda}{\hat{d}}v$ is primitive in $L_2$. Moreover, the matrix $A$ of the lattice inclusion from $L_1$ to $L_2$ in the basis $(v,u)$ and $(\frac{\lambda}{\hat{d}}v,u_2)$
satisfies $\gcd(A,e)=1$.
\end{lemma}

\begin{proof}
By Theorem~\ref{thm: eigenform locus}, $L_1$ and $L_2$ are isogenous. Suppose that $v$ is the holonomy of $\alpha$, we have
\begin{equation}
    \lambda v =\lambda \int_{\alpha} \omega_1 = \int_\alpha p^*(\omega_2) =\int_{p_* \alpha} \omega_2.
\end{equation}
Then $\lambda v$ is the holonomy vector of a closed curve $p_*\alpha \in H_1(T_2,\Z)$, which implies $\lambda v\in L_2$.

Define the divisibility of $\lambda v$ in $L_2$ by
\begin{align*}
\widehat{d}
=
\max\{k\ge1:\lambda v\in kL_2\}.
\end{align*}
Then
$v_2=\frac{\lambda v}{\widehat{d}}$
is primitive in $L_2$ by definition.
Since $v$ is primitive, choose $u\in L_1$ such that
$L_1=\mathbb Zv\oplus\mathbb Zu.$
Since $v_2$ is primitive, choose $u_2\in L_2$ such that
$L_2=\mathbb Zv_2\oplus\mathbb Zu_2$.
Choose both bases compatibly with the orientations of the lattices.
Because $\lambda u\in L_2$, there are integers $c_v,m_v$ such that
$\lambda u=c_vv_2+m_vu_2$.
Therefore the matrix of the lattice inclusion
$p:L_1\longrightarrow L_2$
in the bases $(v,u)$ and $(v_2,u_2)$ is
\begin{align*}
A_v=
\begin{pmatrix}
\widehat d&c_v\\
0&m_v
\end{pmatrix}.
\end{align*}
Its determinant is the index of $\lambda L_1 \subset L_2$. Hence
$d=[L_2:\lambda L_1]
=\det A_v
=\widehat d\,m_v.$
It follows that
\begin{align*}
\widehat d\mid d,
\qquad
m_v=\frac d{\widehat d}.
\end{align*}

Finally, let $A_0$ be a matrix for the original isogeny with respect to any integral bases. Theorem~\ref{thm: eigenform locus} implies that
$\gcd(A_0,e)=1$.
The matrix $A_v$ is obtained from $A_0$ by unimodular changes of bases:
$A_v=Q^{-1}A_0P$ where
$P,Q\in GL_2(\mathbb Z)$. Therefore we have
$\gcd(A_v,e)
=
\gcd(A_0,e)=1.$
So we obtain that
$\gcd\!\left(
\widehat d,c_v,\frac d{\widehat d},e
\right)=1.$
\end{proof}

\begin{notation}
Let $\covol(\Lambda)$ denote the volume of $\R^2/\Lambda$ with respect to the standard Lebesgue measure. 
\end{notation}

\begin{lemma}\label{lem:new splitting}
Let $(X,\omega)\in\mathcal E_D$ and $(L_1,L_2,w)$ be a splitting of type $(d,\lambda)$. Let $v_1\in L_1$ be a primitive vector with 
$0<{w\times v_1}<\min\big(\operatorname{Area}(T_1), \frac{1}{\lambda} \operatorname{Area}(T_2)\big)$.
Then there exists a primitive vector $v_2\in L_2$ such that the splitting obtained by shearing along $v_1$ and $v_2$,
$(L_1^s,L_2^s,w_s)$,
is again a splitting of type $(d,\lambda)$ for every shear parameter s for which the Dehn twist operation is defined. (See Definition~\ref{def:(L_1^s,L_2^s)}.) Moreover, the Dehn twist operation does not change the covolume; namely, $\covol L_1^s=\covol L_1$ and $\covol L_2^s =\covol L_2$.
\end{lemma}
\begin{proof}
Suppose that $L_1=\Z v_1 \oplus \Z u_1$.
Then for $v_1\in L_1$, we can choose $v_2=\frac{\lambda}{\hat{d}}v_1$ by Lemma~\ref{lem:v_2}. 
Then $(v_1,v_2)$ satisfies the conditions (\ref{eq:orientation}) and (\ref{eq:cylinder}). Since $v_2$ is a positive scalar multiple of $v_1$, we have
$(w_k\times v_1)(w_k\times v_2)$ is positive for every integer $k$. Following Lemma~\ref{lem:CM}, we are free to choose the shear parameter $s$. Then $(L_1^s,L_2^s,w_s)$ is a splitting of $(X,\omega)$ for $w_s=w+sv_1+sv_2$ is the holonomy vector of a splitting saddle connection.  
By the type of splitting $(L_1,L_2,w)$,
we have $\lambda u_1=c_v\,v_2+\frac d{\widehat d}\,u_2$ for some $c_v\in \Z$.
After the Dehn twist operation we have
$$
L_1^s=\mathbb Z v_1\oplus
\mathbb Z\left(u_1+s\,v_2\right),\qquad L_2^s=\mathbb Z \frac{\lambda}{\hat d}v_1\oplus
\mathbb Z\left(u_2+s\,v_1\right).
$$
Using $\lambda^2=e\lambda+d$,
\begin{align*}
\lambda v_2
=
\frac{\lambda^2}{\widehat d}v_1
=
e v_2+\frac d{\widehat d}\,v_1.
\end{align*}
Therefore 
\begin{align*}
\lambda(u_1+s v_2)
&=
c_v\,v_2+\frac d{\widehat d}\,u_2
+s\left(e v_2+\frac d{\widehat d}\,v_1\right)\\
&=
(c_v+se)v_2+\frac d{\widehat d}(u_2+s v_1).
\end{align*}
So we have $\lambda L_1^s \subset L_2^s$. Also by Lemma~\ref{lem:v_2}, we can deduce that $\gcd\!\left(
\widehat d,c_v+se,\frac d{\widehat d},e
\right)
=
\gcd\!\left(
\widehat d,c_v,\frac d{\widehat d},e
\right)=1.$
By calculating the covolume of $L_1^s$ and $L_2^s$, we have $\covol(L_1^s)=\covol(L_1)=\frac{d}{\lambda^2}\covol(L_2)=\frac{d}{\lambda^2}\covol(L_2^s)$. Then one can deduce that $$[L_2^s:\lambda L_1^s]=d.$$
So by definition the splitting $(L_1^s,L_2^s,w_s)$ is of type $(d,\lambda)$.
\end{proof}

Let $(X,\omega)\in\mathcal E_D$ and $(L_1,L_2,w)$ be a splitting of type $(d,\lambda)$. We choose $u,v\in L_1$ such that
$$\Delta:=\covol(L_1)=v\times u, \qquad
L_1=\Z v \oplus \Z u.$$
Replace $u$ by $u+mv$ so that the component of $u$ parallel to $v$ has absolute value at most $\frac{\norm{v}}{2}$. Then
\begin{align}\label{def:u}
\norm{u}\le \frac{\norm{v}}{2}+ \frac{\Delta}{\norm{v}}.    
\end{align}
For each $v$, by Lemma~\ref{lem:v_2},
let \(\widehat d(v)\mid d\) be chosen
so that $\frac{\lambda}{\widehat d(v)}v$ is primitive in \(L_2\), and put
\[
a(v):=\frac{\lambda}{\widehat d(v)},
\qquad
\eta(v):=1+a(v).
\]
For \(s\in\{1,2,3\}\), define
\begin{align}\label{eq:child-slit}
w_s:=w+s\eta(v)v
\end{align}
and
\begin{align}\label{ali:L_1'}
L_1^s:=\mathbb Zv\oplus
\mathbb Z\bigl(u+sa(v)v\bigr).
\end{align}

In this case, we say that $w_s$ and $L_i^s$ are obtained by applying a Dehn twist in the direction of $v$ with shear parameter $s$.

\subsection{Decagon curve in $\mathcal{H}(1,1)$}
The following theorem describes the only primitive Teichm\"uller curve generated by a form with simple zeros.

\begin{theorem}\cite[Theorem 1.1]{McM2}\label{thm:decagon}
The decagon form $\omega=d x / y$ on the curve $y^2=x\left(x^5-1\right)$ generates the only primitive Teichm\"uller curve $f: V \rightarrow \mathcal{M}_2$ coming from a form with simple zeros.
\end{theorem}

We will use the following proposition due to Bainbridge, Smillie and Weiss.
\begin{proposition}\cite[Proposition 7.2]{BSW}\label{prop:open set}
For any $M \in \mathcal{E}_D^{(1)}(1,1)$ there is a neighborhood $\mathcal{U}$ of the identity in $L$ and a neighborhood $\mathcal{U}^{\prime}$ of $M$ in $\mathcal{E}_D^{(1)}(1,1)$ such that the map $p: \mathcal{U} \rightarrow \mathcal{U}^{\prime}$ defined by
$$
p(g, v)= \operatorname{Rel}_v(gM)
$$
is the restriction of an affine homeomorphism to $\mathcal{U}$.
\end{proposition}

\begin{remark} \label{rmk:reduction to short slit}
By \Cref{thm:decagon} of McMullen, the only primitive Teichm\"uller curve in $\mathcal{H}(1,1)$ is generated by the regular decagon and it lies in $\mathcal{E}_5$. On the other hand, if $D$ is a nonsquare, then $\mathcal{E}_D$ does not contain non-primitive Teichm\"uller curves.
Hence for any $X\in\mathcal{E}_D$ where $D$ is a nonsquare, unless $D=5$ and $X$ lies on the $\GL_2^+(\R)$-orbit through the regular decagon, the orbit $\GL_2^+(\R)\cdot X$ is dense in $\mathcal{E}_D$ by the theorem of McMullen \cite{McM}; see also later generalization by Eskin-Mirzakhani-Mohammadi \cite{EMM}. By Proposition~\ref{prop:open set}, we have $\GL_2^+(\R)\cdot \{ M_{D,x,y} : \norm{(x,y)}_\infty < \varepsilon \}$ contains an open subset of $\mathcal{E}_D$. So there exists $M_{D,x,y}\in\GL_2^+(\R)\cdot X$ such that $0<\norm{(x,y)}_\infty<\varepsilon$.
\end{remark}

\section{Tree structure of splittings}\label{sec:Tree}
In this section, our objective is to construct a tree structure $\mathcal{T}(r)$ whose nodes are splittings of $(X,\omega) \in \mathcal{E}_D$ where $D$ is nonsquare. We require that any sequence of splittings $\left\{(L_1^n,L_2^n,w^n)\right\}_n$ formed by a path starting from the root satisfies the conditions of Theorem~\ref{thm:nonergodic direction}. Also, in this section we assume that these splittings are all of type $(d,\lambda)$. $d$ and $\lambda$ are given by Theorem~\ref{thm: eigenform locus}.

\begin{theorem}\label{thm:nonergodic direction}
Let $(L_1^{(n)}, L_2^{(n)}, w^{(n)})$  be a sequence of splittings obtained via the Dehn twist construction with shear parameter $s\le 3$. Assume that the directions of the vectors $w^n$ converge to some direction $\theta$. Let $h_n>0$ be the component of $w^{(n)}$ in the direction perpendicular to $\theta$ and $C_i^{(n)}\subset T_i^{(n)}$ be the cylinder contained in each torus $T_i^{(n)}$.
If
\begin{enumerate}
\item $\sum_{n=1}^{\infty} 
\norm{w^{(n)}\times v_1^{(n)}}<\infty$.
\item There exists $c>0$ such that $\operatorname{Area}\left(C_1^{(n)}\right)>c$, $ \operatorname{Area}\left(C_2^{(n)}\right)>c$ for all $n$.
\item $\lim _{n \rightarrow \infty} h_n=0$.
\end{enumerate}
Then $\theta$ is a nonergodic direction.
\end{theorem}

\begin{proof}
Let $\left\{...,C_2^{(n-1)},C_1^{(n)},C_2^{(n)},C_1^{(n+1)},...\right\}$ be the sequence of metric cylinders in Theorem~\ref{thm:MS}.
Since $C_1^{(n)}\subset T_1^{(n)}$ and $C_2^{(n)}\subset T_2^{(n)}$, we have $C_1^{(n)}\cap C_2^{(n)}=\emptyset$. Since $C_2^{(n)}\subset T_2^{(n+1)}$ and $C_1^{(n+1)}\subset T_1^{(n+1)}$, we have $C_2^{(n)}\cap C_1^{(n+1)}=\emptyset$.

Now we have
\begin{align*}
   \mu\left(A_{2n-1} \Delta A_{2n}\right)+\mu\left(A_{2n} \Delta A_{2n+1}\right)=
   \operatorname{Area}\big(X\setminus (C_1^{(n)}\cup C_2^{(n)})\big)+\operatorname{Area}\big(X \setminus (C_2^{(n)}\cup C_1^{(n+1)})\big).
\end{align*}
We have chosen $v_2^{(n)}=\frac{\lambda}{\hat{d}} v_1^{(n)}$ where $\hat{d}\ge 1$ depends on $v_1^{(n)}$ and $d$. Then
\begin{align*}
    &\mu\left(A_{2n-1} \Delta A_{2n}\right)+\mu\left(A_{2n} \Delta A_{2n+1}\right) \\
    &\le \norm{w^{(n)} \times v_1^{(n)}}+\norm{w^{(n)} \times v_2^{(n)}}+ \norm{w^{(n+1)}\times v_2^{(n)}} + \norm{w^{(n+1)}\times v_1^{(n+1)}}\\
      &\le 4(1+\lambda)\norm{w^{(n)}\times v_1^{(n)}} + \norm{w^{(n+1)}\times v_1^{(n+1)}}.
\end{align*}
Hence the theorem follows from \Cref{thm:MS}.
\end{proof}

\begin{notation}
For any lattice $\Lambda$ in $\R^2$, let $\lambda_1(\Lambda),\lambda_2(\Lambda)$ denote the first and second successive minimum of $\Lambda$ respectively. 
\end{notation}

\subsection{Inductive step}
In this subsection, the main objective is to prove Theorem~\ref{th:induction step}. We show that if $w$ is an $H$-normal slit (see Definition~\ref{def:H-normal}), then we can find at least $c_D \frac{\norm{w}^{r-1}}{\log \norm{w}}$ slits $w'=w+s(1+\frac{\lambda}{\hat d(v)})v$ that are $H$-normal.

\begin{lemma}[Primitive lattice points in a square]\label{lem: lambda_1}
Fix $\kappa\geq 1$. There exist constants $A_0=A_0(\kappa)>1$
and $C_0=C_0(\kappa)>1$ with the following property.

Let $L\subset \mathbb R^2$ be a lattice, let
$\Delta=\covol(L)$, and let $K\subset\mathbb R^2$ be a square of side
length $a$. Assume that
\[
        K\subset B(0,\kappa a)
\]
and that
\[
        \lambda_1(L)\geq A_0\frac{\Delta}{a}.
\]
Then
\[
 \frac{a^2}{C_0\Delta}
 \leq
 \#\{v\in L\cap K:\ v\text{ is primitive in }L\}
 \leq
 \frac{C_0a^2}{\Delta}.
\]
Here a vector $v\in L$ is called primitive if $v\notin mL$ for every
integer $m\geq 2$.
\end{lemma}

\begin{proof}
We first record an elementary fact about primitive integer points.

\medskip

\noindent
\textbf{Claim.}
For every $B\geq 1$ there exist constants $N_0=N_0(B)$ and
$c_B>0$ such that the following holds. Let $N_1,N_2\geq N_0$ be integers
and let $m_0,n_0\in\mathbb Z$ satisfy
\[
        |m_0|\leq B N_1,\qquad |n_0|\leq B N_2.
\]
Then
\[
\#\{(m,n)\in\mathbb Z^2:\ |m-m_0|\leq N_1,\ |n-n_0|\leq N_2,
\ \gcd(m,n)=1\}
        \geq c_B N_1N_2 .
\]

Indeed, let
\[
        I=[m_0-N_1,m_0+N_1]\cap\mathbb Z,\qquad
        J=[n_0-N_2,n_0+N_2]\cap\mathbb Z .
\]
By interchanging the two coordinates if necessary, we may assume
$N_2\leq N_1$. For $n\neq 0$, the number of $m\in I$ which are not
coprime to $n$ is at most
\[
        \sum_{p\mid n}\#(I\cap p\mathbb Z),
\]
where the sum is over primes. Since $|n|\leq (B+1)N_2$ for $n\in J$, we get
\begin{align*}
&\#\{(m,n)\in I\times J:\gcd(m,n)>1\}        \\
&\qquad\leq |I|
+\sum_{p\leq (B+1)N_2}
        \#(I\cap p\mathbb Z)\#(J\cap p\mathbb Z)                  \\
&\qquad\leq |I|
+\sum_{p\leq (B+1)N_2}
        \left(\frac{|I|}{p}+1\right)
        \left(\frac{|J|}{p}+1\right)                              \\
&\qquad\leq
        |I||J|\sum_p\frac1{p^2}
        +(|I|+|J|)\sum_{p\leq (B+1)N_2}\frac1p
        +\pi((B+1)N_2)+|I|.
\end{align*}
Since
\(
        \sum_p\frac1{p^2}<\frac12,
\)
we may choose $N_0(B)$ sufficiently large so that, whenever
$N_2\geq N_0(B)$, the last three terms are at most
$\frac14 |I||J|$. Thus at least a fixed positive proportion of the
points of $I\times J$ are primitive. This proves the claim.

\medskip

We now prove the lemma. Let $\Delta=\covol(L)$. Choose a reduced
basis $b_1,b_2$ of $L$. Thus
\[
        \|b_1\|=\lambda_1(L),\qquad
        \|b_1\|\leq \|b_2\|,
\]
and, if $\alpha$ is the angle between $b_1$ and $b_2$, then
\(
        \frac{\pi}{3}\leq \alpha\leq \frac{2\pi}{3},
\)
and thus
\(
        \sin\alpha\geq \frac{\sqrt3}{2}
\)
and
\(
        \Delta=|b_1\times b_2|
        =\|b_1\|\|b_2\|\sin\alpha .
\)
The hypothesis gives
\[
        \|b_2\|
        \leq \frac{2\Delta}{\sqrt3\,\|b_1\|}
        \leq \frac{2a}{\sqrt3\,A_0}.
\]
Since $\|b_1\|\leq\|b_2\|$, both basis vectors have length
$\ll a/A_0$.

We first prove the upper bound. If
\[
        v=mb_1+nb_2\in L\cap K,
\]
then $\|v\|\leq \kappa a$. Since
\[
        m=\frac{v\times b_2}{b_1\times b_2},
        \qquad
        n=\frac{b_1\times v}{b_1\times b_2},
\]
we have
\[
        |m|\leq \frac{2\kappa a}{\sqrt3\,\|b_1\|},
        \qquad
        |n|\leq \frac{2\kappa a}{\sqrt3\,\|b_2\|}.
\]
Therefore
\[
        \#(L\cap K)
        \ll_\kappa
        \frac{a}{\|b_1\|}\frac{a}{\|b_2\|}
        \leq
        \frac{C_\kappa a^2}{\Delta},
\]
because
\[
        \frac1{\|b_1\|\|b_2\|}
        =
        \frac{\sin\alpha}{\Delta}
        \leq \frac1{\Delta}.
\]
This gives the desired upper bound, since primitive points are a subset
of all lattice points.

It remains to prove the lower bound. Let $z$ be the center of $K$.
Since $K\subset B(0,\kappa a)$, we have $\|z\|\leq \kappa a$. The
covering radius of the fundamental parallelogram spanned by $b_1,b_2$ is
at most
\[
        \frac{\|b_1\|+\|b_2\|}{2}.
\]
Hence there exists a lattice point
\[
        x_0=m_0b_1+n_0b_2\in L
\]
such that
\[
        \|x_0-z\|
        \leq \frac{\|b_1\|+\|b_2\|}{2}.
\]
Choose
\[
        N_i=\left\lfloor \frac{a}{10\|b_i\|}\right\rfloor,
        \qquad i=1,2.
\]
We choose $A_0=A_0(\kappa)$ sufficiently large so that
\[
        N_i\geq N_0(B_\kappa)
\]
for $i=1,2$, where
\[
        B_\kappa=\frac{40(\kappa+1)}{\sqrt3}.
\]
For every pair of integers $(m,n)$ satisfying
\[
        |m|\leq N_1,\qquad |n|\leq N_2,
\]
we have
\[
\begin{aligned}
\|x_0+mb_1+nb_2-z\|
&\leq \|x_0-z\|+N_1\|b_1\|+N_2\|b_2\|  \\
&\leq \frac{\|b_1\|+\|b_2\|}{2}+\frac a{10}+\frac a{10}.
\end{aligned}
\]
Increasing $A_0$ if necessary, the first term is at most $a/20$.
Therefore
\[
        \|x_0+mb_1+nb_2-z\|\leq \frac a4<\frac a2.
\]
Since the Euclidean ball of radius $a/2$ centered at $z$ is contained in
the square $K$, it follows that
\[
        x_0+mb_1+nb_2\in K
\]
for all such $(m,n)$.

We now check that the center coordinates $m_0,n_0$ are not too large.
Using again $\sin\alpha\geq \sqrt3/2$, we obtain
\[
        |m_0|
        =
        \left|\frac{x_0\times b_2}{b_1\times b_2}\right|
        \leq
        \frac{2\|x_0\|}{\sqrt3\,\|b_1\|},
\]
and similarly
\[
        |n_0|
        \leq
        \frac{2\|x_0\|}{\sqrt3\,\|b_2\|}.
\]
Since
\[
        \|x_0\|
        \leq \|z\|+\|x_0-z\|
        \leq (\kappa+1)a
\]
after increasing $A_0$ if necessary, and since
\[
        N_i\geq \frac{a}{20\|b_i\|},
\]
we get
\[
        |m_0|\leq B_\kappa N_1,
        \qquad
        |n_0|\leq B_\kappa N_2.
\]
By the claim, the number of pairs $(m,n)$ with
\[
        |m|\leq N_1,\qquad |n|\leq N_2
\]
for which
\[
        \gcd(m_0+m,n_0+n)=1
\]
is at least
\[
        c_{B_\kappa}N_1N_2.
\]
For each such pair, the vector
\[
        x_0+mb_1+nb_2
        =
        (m_0+m)b_1+(n_0+n)b_2
\]
is primitive in $L$ and lies in $K$.

Finally,
\[
        N_1N_2
        \geq
        \frac{a^2}{400\|b_1\|\|b_2\|}
        =
        \frac{a^2\sin\alpha}{400\Delta}
        \geq
        \frac{\sqrt3}{800}\frac{a^2}{\Delta}.
\]
Thus
\[
        \#\{v\in L\cap K:\ v\text{ is primitive in }L\}
        \geq
        c_{B_\kappa}\frac{\sqrt3}{800}\frac{a^2}{\Delta}.
\]
After increasing $C_0(\kappa)$ to dominate the constants in the upper and
lower estimates, the lemma follows.
\end{proof}

Henceforth, set $\kappa=4$. Consequently, the constants $A_0$ and $C_0$ in Lemma~\ref{lem: lambda_1} may be taken to be universal.
\begin{notation}
Fix \(1<r<2\). For a nonzero vector \(y\in\mathbb R^2\) with $\norm{y}>1$, define
\[
P_y(x):=\frac{\langle x,y\rangle}{\|y\|},
\qquad
T_y(x):=\frac{y\times x}{\|y\|},
\]
and let
\[
U_yx:=\bigl(T_y(x),P_y(x)\bigr).
\]
Thus \(U_y\) is an orthogonal linear map with
\(
|\det U_y|=1.
\)

Put
\[
R_y:=\|y\|^r,
\qquad
Q_y:=\|y\|\log\|y\|,
\]
and define
\[
g_y:=
\begin{pmatrix}
\sqrt{Q_yR_y} & 0\\
0 & (Q_yR_y)^{-1/2}
\end{pmatrix}
U_y.
\]
\end{notation}

\begin{definition}[H-normal slit]\label{def:H-normal}
Let $(L_1,L_2,w)$ be a splitting and $\Delta=\covol(L_1)$. For $H>0$, we say $w$ is an $H$-normal slit, or $(L_1,L_2,w)$ is an $H$-normal splitting, if 
\begin{align*}
    \lambda_1(g_w \cdot L_1) \geq H \Delta \sqrt{\frac{Q_w}{R_w}}.
\end{align*}
\end{definition}

\begin{definition}[Good twists]\label{def:good twist}
Let $(L_1,L_2,w)$ be a splitting and $\Delta=\covol(L_1)$. For $H>0$, a primitive vector $v \in L_1$ is said to be an $H$-good twist if the following two conditions hold:

(i) $R_w\leq  P_w(v) \leq \frac{3}{2}R_w$ and $\frac{1}{4Q_w}\le T_w(v) \le \frac{3}{4Q_w}$;

(ii) At least one of $(L_1^s,L_2^s,w_s)$, for $s\in \left\{1,2,3\right\}$, is $H$-normal. More concretely,
for $g_{w_s}=\left( \begin{smallmatrix} (Q_{w_s} R_{w_s})^{1/2} &0\\ 0 &(Q_{w_s} R_{w_s})^{-1/2} \end{smallmatrix} \right) U_{w_s}$, then
$\lambda_1(g_{w_s} \cdot L_1^s)\ge H\Delta \sqrt{\frac{Q_{w_s}}{R_{w_s}}}$ hold for at least one of $s=1,2,3$, where $w_s$ and $L_1^s$ are defined in \eqref{eq:child-slit} and \eqref{ali:L_1'}.
\end{definition}

Let
\[
\mathcal A_D
:=
\left\{
\frac{\lambda}{\widehat d}:
\widehat d\mid d
\right\},
\qquad
N_D:=\#\mathcal A_D.
\]
In particular, \(N_D\leq \tau(d):=\#\{\delta\in\mathbb Z_{>0}:\delta\mid d\}\).

\begin{lemma}[Bad-child estimate]
\label{lem:bad-child-estimate}
Let $(L_1,L_2,w)$ be a splitting of type $(d,\lambda)$ and $    \operatorname{covol}(L_1)=d$. Then there exists
\[
H_D\geq
\max\left\{
2A_0,\,
64C_0N_D
\right\}
\]
such that the following holds.

For every \(H\geq H_D\), if $w$ is an $H$-normal slit,
let $$\mathcal V(w)=
\left\{
v\in L_1^{\mathrm{pr}}:
R_w\le P_w(v)\le\frac32R_w,\quad
\frac1{4Q_w}\le T_w(v)\le\frac3{4Q_w}
\right\}$$ be the set of the primitive vectors of $L_1$ which satisfy condition (i) in Definition~\ref{def:good twist}.
Let 
\begin{align}\label{eq:bad-set-definition}
\mathcal B_H(w)
:=
\left\{
v\in\mathcal V(w):
\lambda_1(g_{w_s}L_1^s)
<
Hd
\sqrt{\frac{Q_{w_s}}{R_{w_s}}}
\text{ for every }s=1,2,3
\right\}
\end{align}
where $w_s$ and $L_1^s$ are defined in (\ref{eq:child-slit}) and (\ref{ali:L_1'}).
Then there is
\(W_0=W_0(D,r,H)\) such that, whenever \(\norm{w}\geq W_0\) 
the following bad set satisfies
\begin{align}\label{eq:bad-child-bound}
\#\mathcal B_H(w)
\leq
\frac{1}{8C_0d}\frac{R_w}{Q_w}.
\end{align}
\end{lemma}

\begin{proof}
Let $W=\|w\|$, $R=W^r$ and $Q=W\log W$.
Fix \(v\in\mathcal B_H(w)\), and abbreviate
\[
a:=\frac{\lambda}{\hat{d}(v)},
\qquad
\eta:=1+a.
\]
All constants $\left\{C_i: 0\le i\le 15\right\}$ below are independent of \(W,v,s\) and the divisor $\hat{d}(v)$ . They may depend on
$D,r$ and $H$.

Since \(v\in\mathcal V(w)\), we have
\begin{align}\label{eq:v-length}
    R\leq \|v\|<2R
\end{align}
for all sufficiently large \(W\). Moreover,
$X:=w\times v
=
W T_w(v)$
satisfies
\begin{align}\label{eq:X-range}
    \frac1{4\log W}
\leq
X
\leq
\frac3{4\log W}.
\end{align}
Since \(a\) belongs to the finite set \(\mathcal A_D\), we have
\begin{align}\label{eq:child-length-comparison}
 R\leq \|w_s\|\leq [1+6(1+\lambda)]R   
\end{align}
for \(s=1,2,3\), once \(W\) is sufficiently large. 

Since \(v\) is bad, for every \(s\in\{1,2,3\}\) there exists a
primitive vector
$\zeta_s\in L_1^s\setminus\{0\}$
such that
\begin{align}\label{eq:short-zeta}
\|g_{w_s}\zeta_s\|<Hd
\sqrt{\frac{Q_{w_s}}{R_{w_s}}}.
\end{align}
 Write
$\zeta_s
=
\alpha_sv+\beta_s(u+sav)
=
(\alpha_s+sa\beta_s)v+\beta_su,$
where $\alpha_s,\beta_s\in\mathbb Z$ and $\gcd(\alpha_s,\beta_s)=1.$
The transverse component in \eqref{eq:short-zeta} gives $\sqrt{Q_{w_s}R_{w_s}}\,
\frac{|w_s\times\zeta_s|}{\|w_s\|}
<
Hd
\sqrt{\frac{Q_{w_s}}{R_{w_s}}}.$
Therefore, we have $|w_s\times\zeta_s|
<
Hd
\frac{\|w_s\|}{R_{w_s}}
=
Hd\|w_s\|^{1-r}.$
By \eqref{eq:child-length-comparison}, there is a constant \(C_1>0\)
such that
\begin{align}
|w_s\times\zeta_s|
\leq
C_1Hd R^{1-r}
=:E_W.
\label{eq:transverse-error}
\end{align}

\medskip
\noindent
\textbf{Claim 1: \(\beta_s\) has a uniform bound.}
Since \(|\det g_{w_s}|=1\), we have
\begin{align}
d|\beta_s|
=
|\zeta_s\times v|
=
|(g_{w_s}\zeta_s)\times(g_{w_s}v)|.
\label{eq:beta-cross-product}
\end{align}
Also, $w_s\times v=w\times v=X.$
If $A_s:=\sqrt{Q_{w_s}R_{w_s}},$
then
\begin{align}
\begin{split}
\|g_{w_s}v\|
&\leq
A_s\frac{|w_s\times v|}{\|w_s\|}
+
A_s^{-1}\|v\| \\
&\leq
C_2
\frac{R^{(r-1)/2}}{\sqrt{\log W}}.
\end{split}
\label{eq:gs-v-bound}
\end{align}
On the other hand, \eqref{eq:child-length-comparison} and \eqref{eq:short-zeta} imply
\begin{align}
\|g_{w_s}\zeta_s\|
\leq
C_3Hd
R^{(1-r)/2}\sqrt{\log W}.
\label{eq:gs-zeta-bound}
\end{align}
Combining \eqref{eq:beta-cross-product},
\eqref{eq:gs-v-bound}, and \eqref{eq:gs-zeta-bound}, we obtain $|\beta_s|\leq C_4H.$
Choose an integer \(B=B(D,r,H,d)\) such that
\begin{align}
|\beta_s|\leq B
\qquad
(s=1,2,3).
\label{eq:beta-bound}
\end{align}

\medskip
\noindent
\textbf{Claim 2: $\alpha_s$ is bounded from above.} 
We show that
\begin{align}
|\alpha_s|\leq C_5H\log W.
\label{eq:alpha-bound}
\end{align}
Indeed, since \(A_s>1\),
$\|\zeta_s\|
\leq
A_s\|g_{w_s}\zeta_s\|
<
Hd Q_{w_s}
\leq
C_6Hd R\log W.$
Write
\[
u=t v+u^\perp,
\qquad
u^\perp\perp v.
\]
By \eqref{def:u} for $1<r<2$, \(|t|\leq 1/2\). Hence the component of
\(\zeta_s\) parallel to \(v\) gives
\[
\left|
\alpha_s+sa\beta_s+t\beta_s
\right|
\|v\|
\leq
\|\zeta_s\|.
\]
Using \eqref{eq:v-length} and \eqref{eq:beta-bound} proves
\eqref{eq:alpha-bound}.

Next we show that
$\beta_s\neq 0$
for $s=1,2,3$.
Indeed, if \(\beta_s=0\), then primitivity implies
\(\alpha_s=\pm1\), so \(\zeta_s=\pm v\). By
\eqref{eq:transverse-error}, we have
$X
=
|w_s\times v|
\leq
C_1Hd R^{1-r},$
contradicting the lower bound in \eqref{eq:X-range} when \(W\) is
sufficiently large.

Set
\begin{align}
    \Xi:=\frac{\alpha_1}{\beta_1}-2\frac{\alpha_2}{\beta_2}+\frac{\alpha_3}{\beta_3}.
\label{eq:Xi-equation}
\qquad
Y:=w\times u.
\end{align}

\medskip
\noindent
\textbf{Claim 3: $\Xi =0$ when $W$ is sufficiently large.}\\
Expanding \(w_s\times\zeta_s\), we obtain
\[
\begin{split}
w_s\times\zeta_s
&=
(w+s\eta v)
\times
\bigl((\alpha_s+sa\beta_s)v+\beta_su\bigr)\\
&=
(\alpha_s+sa\beta_s)X
+\beta_sY
+s\eta\beta_sd.
\end{split}
\]
Since \(|\beta_s|\geq1\), equation
\eqref{eq:transverse-error} gives
\begin{align}
\left|
Y+(\frac{\alpha_s}{\beta_s}+sa)X+s\eta d
\right|
\leq E_W
\label{eq:three-equations}
\end{align}
for \(s=1,2,3\).

Subtracting the equation for \(s=1\) from the equation for \(s=2\),
and then the equation for \(s=2\) from the equation for \(s=3\),
gives
\begin{align}
\left|
(\frac{\alpha_2}{\beta_2}-\frac{\alpha_1}{\beta_1}+a)X+\eta d
\right|
\leq 2E_W,
\label{eq:first-difference}
\end{align}
and
\begin{align*}
\left|
(\frac{\alpha_3}{\beta_3}-\frac{\alpha_2}{\beta_2}+a)X+\eta d
\right|
\leq 2E_W.
\end{align*}
Subtracting these two inequalities yields $|\Xi X|\leq4E_W$.
By \eqref{eq:beta-bound}, every \(\frac{\alpha_s}{\beta_s}\) has denominator at most
\(B\). Therefore either \(\Xi=0\), or $|\Xi|\geq B^{-3}$.
If \(\Xi\neq0\), then \eqref{eq:Xi-equation} gives $X\leq4B^3E_W$.
This again contradicts \eqref{eq:X-range} when \(W\) is sufficiently
large. Consequently, $\Xi=0$.

Then define
\begin{align*}
T:=\frac{\alpha_2}{\beta_2}-\frac{\alpha_1}{\beta_1}=\frac{\alpha_3}{\beta_3}-\frac{\alpha_2}{\beta_2}.
\end{align*}
Then \eqref{eq:first-difference} becomes
\begin{align}
\left|
(T+a)X+\eta d
\right|
\leq2E_W.
\label{eq:resonance-equation}
\end{align}
Put
\begin{align}
b:=\frac{\alpha_1}{\beta_1}-T=2\frac{\alpha_1}{\beta_1}-\frac{\alpha_2}{\beta_2}.
\label{eq:b-definition}
\end{align}
Subtracting \eqref{eq:first-difference} from the \(s=1\) instance of
\eqref{eq:three-equations}, we obtain
\begin{align}
|Y+bX|=|w\times( u+ bv)|\leq3E_W.
\label{eq:Y-bX}
\end{align}

\medskip
\noindent
\textbf{Claim 4: $u+bv$ determines a primitive vector $z(v)$ in $L_1$.}

Write
$b=\frac pq$
in lowest terms, with \(q>0\). Since
\(b=2\frac{\alpha_1}{\beta_1}-\frac{\alpha_2}{\beta_2}\) and the denominators of \(\frac{\alpha_1}{\beta_1},\frac{\alpha_2}{\beta_2}\) are at most \(B\),
we have
\begin{align}
1\leq q\leq B^2.
\label{eq:q-bound}
\end{align}
Define $z(v):=qu+pv\in L_1$.
Because \(\gcd(p,q)=1\) and \((v,u)\) is a basis of \(L_1\), the
vector \(z(v)\) is primitive in \(L_1\).
Equation \eqref{eq:Y-bX} gives
\begin{align}
|w\times z(v)|
=
q|Y+bX|
\leq
3B^2E_W.
\label{eq:z-transverse}
\end{align}
Moreover,
\begin{align}
z(v)\times v
=
q(u\times v)
=
-qd.
\label{eq:z-cross-v}
\end{align}
By \eqref{eq:beta-bound}, \eqref{eq:alpha-bound}, and
\eqref{eq:b-definition}, we have $|b|\leq C_7H\log W$.
Together with \eqref{eq:q-bound}, this gives $|p|\leq C_8H\log W$.
The definition of $u$ and
\(v\times u=d\) imply $\|u\|
\leq
\frac12\|v\|+\frac{d}{\|v\|}
\leq C_9R$.
Therefore
\begin{align}
\|z(v)\|\leq C_{10}R\log W.
\label{eq:z-length}
\end{align}

\medskip
\noindent
\textbf{Claim 5: the vectors $z(v)$ agree up to sign.}

Let \(v,v'\in\mathcal B_H(w)\), and put
$z:=z(v)$ and $z':=z(v')$.
Using coordinates parallel and perpendicular to \(w\), we obtain
$|z\times z'|
\leq
\frac{
\|z\|\,|w\times z'|
+
\|z'\|\,|w\times z|
}{W}$.
By \eqref{eq:z-transverse} and \eqref{eq:z-length}, we have $|z\times z'|
\leq
C_{11}
\frac{R\log W\,R^{1-r}}{W}$.
Since \(R=W^r\), $\frac{R^{2-r}}W
=
W^{r(2-r)-1}
=
W^{-(r-1)^2}.$
Thus
\begin{align*}
|z\times z'|
\leq
C_{11}W^{-(r-1)^2}\log W.
\end{align*}
For sufficiently large \(W\), the right-hand side is less than
\(d\). Since \(z,z'\in L_1\), $z\times z'\in d\mathbb Z$.
It follows that $z\times z'=0.$
Both \(z\) and \(z'\) are primitive in \(L_1\), and therefore
$z'=\pm z.$

Consequently, there exists a primitive vector \(z_0\in L_1\) such
that
\begin{align*}
z(v)\in\{z_0,-z_0\}
\qquad
\text{for every }v\in\mathcal B_H(w).
\end{align*}

\medskip
\noindent
\textbf{Claim 6: for fixed \(a\) and a fixed sign of \(z_0\), the integer \(q\)
in \eqref{eq:z-cross-v} is fixed.}

For \(z=qu+pv\), using \eqref{eq:z-cross-v} and the identity
\[
z\times v
=
P_w(z)T_w(v)-T_w(z)P_w(v),
\]
we obtain
\begin{align}
\frac{P_w(z)}WX
=
-qd
+
\frac{w\times z}{W}P_w(v).
\label{eq:cX-equation}
\end{align}
By \eqref{eq:z-transverse} and
\(P_w(v)\leq 3R/2\), we have $\left|
\frac{P_w(z)}WX+qd
\right|
\leq
C_{12}W^{-(r-1)^2}.$
Multiplying \eqref{eq:cX-equation} by \(\eta\), multiplying
\eqref{eq:resonance-equation} by \(q\), and subtracting, we get
\[
\left|
\eta \frac{P_w(z)}W-q(T+a)
\right|
X
\leq
C_{13}W^{-(r-1)^2}.
\]
Using the lower bound for \(X\) in \eqref{eq:X-range}, we conclude
that
\begin{align}
\left|
\frac{P_w(z)}W-\frac{q(T+a)}{1+a}
\right|
\leq
\varepsilon_W,
\label{eq:resonance-approximation}
\end{align}
where
\begin{align*}
\varepsilon_W
:=
C_{14}W^{-(r-1)^2}\log W
\longrightarrow0.
\end{align*}

We now prove that the possible values in
\eqref{eq:resonance-approximation} are uniformly separated.
For fixed \(a\in\mathcal A_D\), define
\begin{align*}
\mathcal F_{a,B}
:=
\left\{
\frac{q(T+a)}{1+a}:
\begin{array}{l}
1\leq q\leq B^2,\\
T\in\mathbb Q,\\
\text{denominator of }T \text{ is at most } B^2
\end{array}
\right\}.
\end{align*}
Now we show that although $\mathcal F_{a,B}$ is infinite, it is uniformly discrete. Let $L_B:=\operatorname{lcm}(1,2,\ldots,B^2)$.
If $\frac{q(T+a)}{1+a}$ and $\frac{q'(T'+a)}{1+a}$
are two elements of \(\mathcal F_{a,B}\), their difference is of the
form $\frac{m/L_B+(q-q')a}{1+a}$
for some \(m\in\mathbb Z\).
If \(q=q'\), a nonzero difference has absolute value at least $\frac{1}{L_B(1+a)}$.
If \(q\neq q'\), then \(q-q'\) belongs to the finite set $\{-B^2+1,\ldots,-1,1,\ldots,B^2-1\}$.
Since \(D\) is nonsquare, \(\lambda\) is irrational, and hence every
\(a\in\mathcal A_D\) is irrational. Therefore $\operatorname{dist}
\left(
(q-q')a,\frac1{L_B}\mathbb Z
\right)>0$.
Since there are only finitely many possible values of \(q-q'\), there
exists $d_{a,B}>0$
such that distinct elements of \(\mathcal F_{a,B}\) are separated by
at least \(d_{a,B}\).
Because \(\mathcal A_D\) is finite, we may put $d_B
:=
\min_{a\in\mathcal A_D}d_{a,B}>0.$
Increase \(W_0\) so that
\begin{align}
2\varepsilon_W<d_B.
\label{eq:epsilon-less-delta}
\end{align}
Fix \(a\in\mathcal A_D\) and fix one of the two choices
\[
z=z_0
\qquad\text{or}\qquad
z=-z_0.
\]
Then \(\frac{P_w(z)}W\) is fixed. By
\eqref{eq:resonance-approximation} and
\eqref{eq:epsilon-less-delta}, there is at most one value
\(
\frac{q(T+a)}{1+a}\in\mathcal F_{a,B}
\)
which can occur.

Moreover, the pair \((q,T)\) is uniquely determined by this value.
Indeed, if
\(
q(T+a)=q'(T'+a),
\)
then
\(
qT-q'T'+(q-q')a=0.
\)
Since \(a\) is irrational, this implies \(q=q'\), and hence
\(T=T'\).
Thus, for fixed \(a\) and a fixed sign of \(z_0\), the integer \(q\)
in \eqref{eq:z-cross-v} is fixed.

\medskip
\noindent
\textbf{Claim 7: $v\in \mathcal B_H(w)$ lies on finitely many affine
lattice lines.}

For fixed \(a\) and fixed \(z=\pm z_0\), equation
\eqref{eq:z-cross-v} has the form 
\begin{align}
z\times v=-qd
\label{eq:affine-line-equation}
\end{align}
for one fixed integer \(q\).

Since \(z\) is primitive, it can be completed to a basis of \(L_1\).
It follows that the set of solutions of
\eqref{eq:affine-line-equation} is either empty or is an affine
lattice line
\(
v_0+\mathbb Zz.
\)
There are at most \(N_D\) choices of \(a\), and two choices for the
sign of \(z_0\). Therefore
\(
\mathcal B_H(w)
\)
is contained in at most
\begin{align}
2N_D
\label{eq:number-lines}
\end{align}
affine lattice lines parallel to \(z_0\).

\medskip
\noindent
\textbf{Claim 8:$\# \mathcal B_H(w)
\leq
\frac{1}{8C_0d}\frac RQ$ for large enough $W$.}

Since \(z_0\) is primitive, the condition $\lambda_1(g_w L_1)\geq Hd \sqrt{\frac{Q}{R}}$ gives
\begin{align}
\|g_wz_0\|^2
=
QR\,T_w(z_0)^2
+
\frac{P_w(z_0)^2}{QR}
\geq
H^2d^2\frac QR.
\label{eq:normality-z0}
\end{align}
By \eqref{eq:z-transverse},
\(
|T_w(z_0)|
=
\frac{|w\times z_0|}{W}
\leq
C_{15}\frac{R^{1-r}}W.
\)
Hence
\begin{align}
QR\,T_w(z_0)^2
\leq
C_{16}d^2\frac QR
W^{-2(r-1)^2}.
\label{eq:transverse-negligible}
\end{align}
After increasing \(W_0\), the right-hand side is at most
\(
\frac14H^2d^2\frac QR.
\)
It follows from \eqref{eq:normality-z0} that
\(
\frac{P_w(z_0)^2}{QR}
\geq
\frac34H^2d^2\frac QR.
\)
In particular, we have $|P_w(z_0)|
\geq
\frac H2d Q.$
Consider one affine line
\(
v_0+\mathbb Zz_0.
\)
Successive points on this line have longitudinal coordinates differing
by
\(
|P_w(z_0)|
\geq
\frac H2d Q.
\)
The candidate interval
\(
R\leq P_w(v)\leq\frac32R
\)
has length \(R/2\). Therefore the number of candidate vectors on one
affine line is at most
\(
1+
\frac{R/2}{(H/2)d Q}
=
1+\frac{R}{Hd Q}.
\)
Since \(R/Q\to\infty\), we may increase \(W_0\) so that $1+\frac{R}{Hd Q}
\leq
\frac{2R}{Hd Q}.$
Then by \eqref{eq:number-lines}, we obtain
\begin{align}
\#\mathcal B_H(w)
\leq
2N_D\cdot
\frac{2R}{Hd Q}
=
\frac{4N_D}{Hd}\frac RQ.
\label{eq:bad-preliminary}
\end{align}
Since
\(
H\geq64C_0N_D,
\)
we have
\(
\frac{4N_D}{H}
\leq
\frac{1}{16C_0}
<
\frac{1}{8C_0}.
\)
Therefore
\[
\#\mathcal B_H(w)
\leq
\frac{1}{8C_0d}\frac RQ,
\]
as required.
\end{proof}

Let $H$ be a constant such that $H>H_D$ where $H_D$ is introduced in Lemma~\ref{lem:bad-child-estimate}.

\begin{theorem}[$H$-normal-child estimate]
\label{th:induction step}
There are constants $c_D=\frac{1}{8 C_0 d},C_D=1+6(1+\lambda)$ and $W_1=\max(W_0, e^{\frac{1}{\min (d,\lambda)}},
e^{1/(r-1)},
\left(\frac{4}{c_D(r-1)}\right)^{\frac{2}{r-1}})$ such that every $H$-normal splitting of $(d,\lambda)$-type satisfying $\covol L_1=d$ and $\|w\|\ge W_1$ has at least $ c_D\frac{\|w\|^{r-1}}{\log\|w\|}$
children $w'$, and every child satisfies:
\begin{itemize}
    \item[(i)] the selected type and covolume are preserved;
    \item[(ii)] $w'$ is $H$-normal;
    \item[(iii)] $\|w\|^r\le \|w'\|\le C_D\|w\|^r$;
    \item[(iv)] $|w\times w'|\le \frac{C_D}{\log\|w\|}$.
\end{itemize}
\end{theorem}
\begin{proof}
Denote $R=\norm{w}^r$ and $Q=\norm{w}\log \norm{w}$.
Let $$\mathcal V(w)=
\left\{
v\in L_1^{\mathrm{pr}}:
R\le P_w(v)\le\frac32R,\quad
\frac1{4Q}\le T_w(v)\le\frac3{4Q}
\right\},$$ then $g_w \mathcal V(w)$ is a square of side $a=\frac{1}{2}\sqrt{\frac{R}{Q}}$. Also it is contained in $B(0,4a)$. Because $w$ is $H$-normal, we have $\lambda_1(g_w L_1)\ge \frac{Hd}{(R/Q)^{1/2}}>2A_0 d\sqrt{\frac{Q}{R}}$. By Lemma~\ref{lem: lambda_1}, the number of the primitive vectors in $\mathcal V(w)$ is at least $\frac{1}{4C_0 d}\cdot \frac{R}{Q}$. 
By Lemma~\ref{lem:bad-child-estimate}, the number of vectors in $\mathcal V(w)$ that are not $H$-good is at most $\frac{1}{8C_0 d}\frac{R}{Q}$ if $\norm{w}>W_0$. Then the number of $H$-good twists $v\in L_1$ is at least $\frac{1}{8C_0 d}\frac{R}{Q}$. Also, because $v\in \mathcal V(w)$, 
$0<\norm{w\times v}<\frac{3}{4\log W_1}<\min (d, \lambda)$. Then by applying Lemma~\ref{lem:new splitting} for some $s$, $(L_1^s,L_2^s,w_s)$ is again a splitting of type $(d,\lambda)$ and covolume which satisfies (i). 

Suppose that $v$ and $v'$ are distinct good twists for $w$ and write
$w_1=w+s \eta(v) v, w_2=w+s'\eta(v')v'$. $w_1=w_2$ implies that
$s \eta(v) v=s'\eta(v')v'$. So $v$ and $v'$ are parallel. But $v$ and $v'$ are primitive vectors in $L_1$, so $v= v'$ ($v=-v'$ is not allowed because $T_w(v)>0$). So $s=s'$ and $\eta(v)=\eta(v')$. So distinct good vectors give distinct children. 
Then the number of $H$-normal slits $w'=w+s(v)(1+\frac{\lambda}{\hat{d}(v)})v$ is at least $$c_D \frac{\norm{w}^{r-1}}{\log \norm{w}}$$ for some constant $c_D=\frac{1}{8 C_0 d}$. By $
W_1\ge
\max(
e^{1/(r-1)},
\left(\frac{4}{c_D(r-1)}\right)^{\frac{2}{r-1}})$, we can deduce that $c_D \frac{W_1^{r-1}}{\log W_1}\ge 2$ which implies the number of children is at least 2. 

Let $C_D=1+6(1+\lambda)$, we have $\norm{w}^r \le \norm{w'} \le C_D \norm{w}^r$. Because $v\in \mathcal V(w)$, we have
\begin{align*}
    |w\times w'|\le s(v)(1+\frac{\lambda}{\hat{d}(v)}) |w\times v| \le \frac{C_D}{\log \norm{w}}.
\end{align*}

\end{proof}

\subsection{Tree structure $\mathcal{T}(r)$ of splittings}
We then represent the collection of splittings through a \textit{graded tree structure} $\mathcal{T}(r)$. This tree captures the hierarchical evolution of splittings under the twisting operation:

\begin{itemize}
    \item \textit{Initial splitting}: In Section~\ref{sec:initial}, we will show the existence of an $H$-normal root $(L_1^{(1)},L_2^{(1)},w^{(1)})$ which satisfies the condition of Theorem~\ref{th:induction step}.
    \item \textit{Nodes and Scales}: A node at level $n\ge 2$, denoted by $S_{w^{(n)}}=(L_1^{(n)},L_2^{(n)},w^{(n)})$, has slit length constrained by the following scale:
    \begin{equation}\label{eq:length}
        \norm{w^{(1)}}^{r^{n-1}} < \norm{w^{(n)}} < C_D^\frac{r^{n-1}-1}{r-1} \norm{w^{(1)}}^{r^{n-1}}.
    \end{equation}
    \item \textit{Branching and Complexity}: 
    An edge exists from $S_w$ to $S_{w'}$ if $w'$ is obtained from $w$ via an $H$-good twisting operation along $v$ (see Definition~\ref{def:good twist}). For each $H$-good twist $v$, we choose only one $s$ from $\left\{1,2,3\right\}$ to obtain a new splitting with slit $w'$. The slit $w'$ by definition is $H$-normal.
    The branching factor $k$ is the number of $H$-good children guaranteed by Theorem~\ref{th:induction step}; it satisfies the density estimate:
    \begin{equation}\label{eq: number}
        k \ge c_D\frac{\norm{w}^{r-1}}{ \log \norm{w}}.
    \end{equation}
\end{itemize}

\section{Initial step}\label{sec:initial}
In this section, we construct the root node of the splitting tree. Let $1<r<2$ and $H$ be a constant such that $H\ge H_D$ where $H_D$ is given in Lemma~\ref{lem:bad-child-estimate}. Then we choose 
\begin{align}\label{def:W*}
   W_*=\max(W_1,e^{\frac{4C_Dr}{r-1}},(2C_D)^{\frac{1}{r^2-r}},4^{\frac{1}{r^2-1}},e^{24(1+\lambda)},C_D^{\frac{2r}{(r-1)^3}}) 
\end{align}
where $W_1$ is given by Theorem~\ref{th:induction step}.

\begin{proposition}[Existence of a normal root of the prescribed type]
\label{prop:normal-root}
Fix \(1<r<2\), \(H>0\), and \(W_*>1\). Then there exists a nonempty
open subset
\[
\mathcal U=\mathcal U(D,r,H,W_*)\subset \mathcal{E}_D
\]
such that every surface \(Y\in\mathcal U\) admits a splitting $(L_1,L_2,w)$
of \((d,\lambda)\)-type satisfying $\norm{w}>W_*$
and
\begin{align}
\lambda_1(g_wL_1)
\geq
H\,d
\sqrt{\frac{Q_w}{R_w}},
\label{eq:root-normality}
\end{align}
where $R_w=\|w\|^r$ and $Q_w=\|w\|\log\|w\|$.

Consequently, if \(X\in \mathcal{E}_D\) has dense
\(\GL_2^+(\mathbb R)\)-orbit, then some surface in the orbit
of \(X\) admits a splitting of \((d,\lambda)\)-type satisfying
\eqref{eq:root-normality}.
\end{proposition}

\begin{proof}
For \(W>1\), let $R_W:=W^r,
Q_W:=W\log W,
A_W:=\sqrt{Q_WR_W}.$
Define $t_W:=\frac{W}{\sqrt d\,A_W}$.
Since $t_W^2
=
\frac{W^2}{dQ_WR_W}
=
\frac{W^{1-r}}{d\log W}$,
we have $t_W\longrightarrow0$ as $W\longrightarrow\infty.$

Choose \(W\) sufficiently large that $W>W_*$, $0<t_W<\min\{\varepsilon,1,\lambda\},$
and
\begin{align}
2Hd\sqrt{\frac{Q_W}{R_W}}<\sqrt d.
\label{eq:strict-root-inequality}
\end{align}
This is possible because $\frac{Q_W}{R_W}
=
W^{1-r}\log W
\longrightarrow0$
as \(W\to\infty\).

Consider the model surface $M_{D,0,t_W}.$
It has a splitting $\bigl(L_1^{(0)},L_2^{(0)},w^{(0)}\bigr),$
where
\[
L_1^{(0)}
=
\mathbb Z(d\mathbf e_1)
\oplus
\mathbb Z\mathbf e_2,
\qquad
L_2^{(0)}
=
\mathbb Z(\lambda\mathbf e_1)
\oplus
\mathbb Z(\lambda\mathbf e_2),
\qquad
w^{(0)}=t_W\mathbf e_2.
\]
Set
\[
v_1^{(0)}:=d\mathbf e_1,
\qquad
u_1^{(0)}:=\mathbf e_2,
\]
and
\[
v_2^{(0)}:=\lambda\mathbf e_1,
\qquad
u_2^{(0)}:=\lambda\mathbf e_2.
\]
Then we have
\[
L_1^{(0)}
=
\mathbb Zv_1^{(0)}
\oplus
\mathbb Zu_1^{(0)}
\qquad
L_2^{(0)}
=
\mathbb Zv_2^{(0)}
\oplus
\mathbb Zu_2^{(0)}.
\]
Moreover, we have $v_2^{(0)}
=
\frac{\lambda}{d}v_1^{(0)}$ and $u_2^{(0)}
=
\lambda u_1^{(0)}.$
Thus the splitting is of \((d,\lambda)\)-type.

Define
\[
h_W
:=
\begin{pmatrix}
(\sqrt d\,A_W)^{-1}&0\\
0&\sqrt d\,A_W
\end{pmatrix}.
\]
Since
\(
\det h_W=1,
\)
we have
\(
h_W\in\mathrm{SL}_2(\mathbb R).
\)
Let
\(
Y_W:=h_W\cdot M_{D,0,t_W}.
\)
Since the eigenform locus is
\(\mathrm{SL}_2(\mathbb R)\)-invariant,
\(
Y_W\in \mathcal{E}_D.
\)

The transformed slit is
\[
\begin{split}
w_W
&=
h_Ww^{(0)}\\
&=
h_W(t_W\mathbf e_2)\\
&=
t_W\sqrt d\,A_W\mathbf e_2\\
&=
W\mathbf e_2.
\end{split}
\]
In particular,
\(
\|w_W\|=W>W_*.
\)
The transformed first lattice is
\[
\begin{split}
L_{1,W}
=
h_WL_1^{(0)}=
\mathbb Z\left(
\frac{\sqrt d}{A_W}\mathbf e_1
\right)
\oplus
\mathbb Z\left(
\sqrt d\,A_W\mathbf e_2
\right).
\end{split}
\]
The transformed second lattice is
\[
\begin{split}
L_{2,W}
=h_WL_2^{(0)}
=
\mathbb Z\left(
\frac{\lambda}{\sqrt d\,A_W}\mathbf e_1
\right)
\oplus
\mathbb Z\left(
\lambda\sqrt d\,A_W\mathbf e_2
\right).
\end{split}
\]
Define
\[
v_{1,W}
:=
\frac{\sqrt d}{A_W}\mathbf e_1,
\qquad
u_{1,W}
:=
\sqrt d\,A_W\mathbf e_2,
\]
and
\[
v_{2,W}
:=
\frac{\lambda}{\sqrt d\,A_W}\mathbf e_1,
\qquad
u_{2,W}
:=
\lambda\sqrt d\,A_W\mathbf e_2.
\]
Then
\[
L_{1,W}
=
\mathbb Zv_{1,W}\oplus\mathbb Zu_{1,W},
\qquad
L_{2,W}
=
\mathbb Zv_{2,W}\oplus\mathbb Zu_{2,W}.
\]
Furthermore,
\begin{align}
v_{2,W}
=
\frac{\lambda}{d}v_{1,W},
\qquad
u_{2,W}
=
\lambda u_{1,W}.
\label{eq:transformed-type-relation}
\end{align}
Thus the transformed splitting
\(
(L_{1,W},L_{2,W},w_W)
\)
is again of \((d,\lambda)\)-type.
We also have
\(
\operatorname{covol}(L_{1,W})
=
\left(
\frac{\sqrt d}{A_W}
\right)
\left(
\sqrt d\,A_W
\right)
=
d.
\)

We now verify the normality condition. Recall that
\(
T_{w_W}(z)
=
\frac{w_W\times z}{\|w_W\|},
\) and
\(
P_{w_W}(z)
=
\frac{\langle z,w_W\rangle}{\|w_W\|}.
\)
Since
\(
w_W=W\mathbf e_2,
\)
we have
\(
T_{w_W}(x_1,x_2)=-x_1,
\) and
\(
P_{w_W}(x_1,x_2)=x_2.
\)
Thus
\(
U_{w_W}
=
\begin{pmatrix}
-1&0\\
0&1
\end{pmatrix}.
\)
The normalizing matrix is
\(
g_{w_W}
=
\begin{pmatrix}
A_W&0\\
0&A_W^{-1}
\end{pmatrix}
U_{w_W}.
\)
Applying \(g_{w_W}\) to the generators of \(L_{1,W}\), we get
\(
g_{w_W}v_{1,W}
=
-\sqrt d\,\mathbf e_1
\)
and
\(
g_{w_W}u_{1,W}
=
\sqrt d\,\mathbf e_2.
\)
Consequently,
\(
g_{w_W}L_{1,W}
=
\sqrt d\,\mathbb Z^2,
\)
up to changing the sign of the first generator. Therefore
\(
\lambda_1(g_{w_W}L_{1,W})
=
\sqrt d.
\)
By \eqref{eq:strict-root-inequality}, we have
\begin{align}
\begin{split}
\lambda_1(g_{w_W}L_{1,W})
&=
\sqrt d\\
&\geq
2Hd\sqrt{\frac{Q_W}{R_W}}\\
&=
2H\operatorname{covol}(L_{1,W})
\sqrt{\frac{Q_W}{R_W}}.
\end{split}
\label{eq:strong-root-normality}
\end{align}
Thus the splitting is \(2H\)-normal and, in particular, \(H\)-normal.

The relations in \eqref{eq:transformed-type-relation}, the slit
splitting, and the strict inequality
\eqref{eq:strong-root-normality} persist under sufficiently small
deformations in the corresponding marked period-coordinate chart.
Equivalently, using the local parameterization from
Proposition~\ref{prop:open set}, the map
\[
(g,z)\longmapsto g\cdot M_{D,z}
\]
maps a neighborhood of
\(
\bigl(h_W,(0,t_W)\bigr)
\)
onto an open neighborhood
\(
\mathcal U\subset \mathcal{E}_D
\)
of \(Y_W\).
After shrinking this neighborhood if necessary, every
\(Y\in\mathcal U\) admits a splitting of \((d,\lambda)\)-type whose
slit has length greater than \(W_*\) and which satisfies
\eqref{eq:root-normality}.

Finally, if \(X\in \mathcal{E}_D\) has dense
\(\GL_2^+(\mathbb R)\)-orbit, then
\(
\GL_2^+(\mathbb R)\cdot X
\)
meets \(\mathcal U\). Hence some surface in the orbit of \(X\) admits
the required normal root splitting, which can replace $X$. Indeed, the Hausdorff dimension of nonuniquely ergodic directions is constant along each $\GL_2^+(\R)$-orbit, as the induced map on $\mathbb{RP}^1$ is bi-Lipschitz.
\end{proof}

\section{Cantor set and Hausdorff dimension 1/2}
In Section~\ref{sec:Tree}, we build a tree $\mathcal{T}(r)$ of splittings of type $(d,\lambda)$ with a parameter $r$. Also, in Section~\ref{sec:initial}, we show that the initial slit $w^{(1)}$ satisfies
\begin{align}\label{eq:initial length}
    \norm{w^{(1)}}> W_*.
\end{align}
In this section, we will first associate $\mathcal{T}(r)$ with a Cantor set $E(r)$, where $E(r)$ is a subset of non-uniquely ergodic directions on $(X,\omega)$. Second, We show that the Hausdorff dimension of $E(r)$ approaches $1/2$ as $r$ tends to 1, which completes the proof of Theorem~\ref{thm:main theorem}.

For a slit $w=(x,y)$, define $\theta(w):=[w]\in \R \mathbb{P}^1$ be the projective direction. 
Also, if $w'=(x',y')$, define
\begin{align*}
\dist_{\mathbb{RP}^1}([w],[w'])
=
\arcsin\!\left(
\frac{\norm{w\times w'}}
{\|w\|\,\|w'\|}
\right).
\end{align*}
For every $t\in [0,1]$, we have $t \le \arcsin t\le \frac{\pi}{2}t$. Let $\left\{ w^{(n)}\right\}$ be a path in the tree $\mathcal T(r)$. We have
\begin{align*}
\frac{\norm{w^{(n)}\times w^{(n+1)}}}{\norm{w^{(n)}}\norm{w^{(n+1)}}}\le \dist_{\mathbb{RP}^1}([w^{(n)}],[w^{(n+1)}]) \le \frac{\pi}{2} \frac{\norm{w^{(n)}\times w^{(n+1)}}}{\norm{w^{(n)}}\norm{w^{(n+1)}}}.
\end{align*}

Let $I(w)=\left\{
\xi\in\mathbb{RP}^1:
\dist_{\mathbb{RP}^1}(\xi,[w])
\leq \frac{1}{2\norm{w}^{r+1}}
\right\}$.
Let
\begin{align*}
F_i=
\bigcup_{\substack{S_w\in\mathcal T(r)\\
\operatorname{level}(S_w)=i}}
I(w).
\end{align*}
be the projective interval union of all slits at level $i$. 

\begin{lemma}\label{lem:subset}
If \(w'\) is a child of \(w\) in the tree $\mathcal{T}(r)$, then
    \begin{align}
        I(w') \subset I(w).
    \end{align}
\end{lemma}
\begin{proof}
Let $w'=w+s(1+\frac{\lambda}{\hat{d}})v$. By Theorem~\ref{th:induction step}, we have
\begin{align*}
    \dist_{\mathbb{RP}^1}([w],[w'])\le \frac{\pi}{2}\frac{|w'\times w|}{\norm{w'}\norm{w}}\le  \frac{\pi C_D}{2\norm{w}^{(1+r)}\log\norm{w}}.
\end{align*}
For any $\theta \in I(w')$,
\begin{align*}
    \dist_{\mathbb{RP}^1}(\theta,\theta(w))\le \frac{1}{2\norm{w'}^{r+1}}+ \frac{\pi C_D}{2\norm{w}^{(1+r)}\log\norm{w}}.
\end{align*}
By (\ref{def:W*}) and (\ref{eq:initial length}), we can deduce that $\frac{\norm{w}^{(r+1)}}{\norm{w'}^{r+1}}<\frac1{W_*^{(r-1)(r+1)}}\le \frac{1}{4}$ and $\frac{\pi C_D}{\log \norm{w}}<\frac{\pi C_D}{\log W_*}\le 
\frac{\pi(r-1)}{4r}<
\frac{\pi}{8}$. Hence $\frac14+\frac{\pi}{8}<1$, which implies
$$\dist_{\mathbb{RP}^1}(\theta,\theta(w))<\frac{1}{2\norm{w}^{(r+1)}}.$$
\end{proof}

\begin{lemma}\label{lem:separation}
Any two distinct projective intervals $I(w')$ and $I(w'')$ in $F(i)$ are separated by a distance greater than $\frac{1}{2\norm{w'}\norm{w''}}$.
\end{lemma}
\begin{proof}
We argue by induction on the level.

Consider the first case: $I(w'),I(w'') \subset I(w)$.
Write $w'=w+s\eta(v')v'$ and $w''=w+s'\eta(v'')v''$ where $\eta(v)=1+\frac{\lambda}{\hat d(v)}$. Here, $v'$ and $v''$ are distinct primitive elements chosen from the same lattice (the case $v''=-v'$ is excluded); it follows that $\norm{v' \times v''} \ge d \ge 1$. Also, $1+\frac{\lambda}{d} \le \eta(v'),\eta(v'')\le 1+\lambda$. Then
 \begin{align*}
     \norm{w'\times w''}&\ge ss'(1+\frac{\lambda}{d})^2\norm{v'\times v''}-(1+\lambda)( s'\norm{w\times v''}+s\norm{v'\times w})\\
                   &> d(1+\frac{\lambda}{d})^2- (1+\lambda)\frac{6}{\log\norm{w}}.
 \end{align*}
Since $\lambda^2=e\lambda+d$, we have
$d\left(1+\frac{\lambda}{d}\right)^2
=d+2\lambda+\frac{\lambda^2}{d}
=d+1+2\lambda+\frac{e\lambda}{d}>2.$
By (\ref{def:W*}) and (\ref{eq:initial length}), we have
$\norm{w}> W_* \ge e^{24(1+\lambda)}$ which gives
$\frac{6(1+\lambda)}{\log\|w\|}\le\frac14.$
Consequently,
$|w'\times w''|>2-\frac14>1.$
It follows that $\dist_{\mathbb{RP}^1}([w'],[w''])>\frac{1}{\norm{w'}\norm{w''}}$. 
For any $x\in I(w'),y\in I(w'')$, we have
\begin{align*}
    \dist_{\mathbb{RP}^1}(x,y)\geq \dist_{\mathbb{RP}^1}([w'],[w''])-\frac{1}{2\norm{w'}^{r+1}}-\frac{1}{2\norm{w''}^{r+1}}.
\end{align*}
By (\ref{def:W*}) and  (\ref{eq:initial length}), we have $\norm{w}>W_*\ge (2C_D)^{1/(r^2-r)}$. Then
\begin{align*}
    \frac{2\max(\norm{w'},\norm{w''})}{\min(\norm{w'},\norm{w''})^r}<\frac{2C_D\norm{w}^r}{\norm{w}^{r^2}}<1
\end{align*}
So
\begin{align*}
    \frac{1}{\norm{w'}^{r+1}}+\frac{1}{\norm{w''}^{r+1}}\le \frac{2}{\min(\norm{w'},\norm{w''})^{r+1}}<\frac{1}{\norm{w'}\norm{w''}}
\end{align*}
which implies that
\begin{align*}
    \dist_{\mathbb{RP}^1}(x,y)>\frac{1}{2\norm{w'}\norm{w''}}.
\end{align*}

In the second case, $I(w')\subset I(w)$ and $I(w'')\subset I(\hat{w})$ where $I(w)$ and $I(\hat{w})$ are two distinct intervals belonging to  $F(i-1)$. By induction,
\begin{align*}
\dist_{\mathbb{RP}^1}
\bigl(I(w),I(\hat w)\bigr)>
\frac{1}{2\|w\|\,\|\hat w\|}.
\end{align*}
Then for $x\in I(w'),\,y\in I(w'')$,
we have $ \dist_{\mathbb{RP}^1}(x,y)>\frac{1}{2\|w\|\,\|\hat w\|}.$
Since child lengths dominate parent lengths, we have
\begin{align*}
    \dist_{\mathbb{RP}^1}(x,y)>\frac{1}{2\norm{w'}\norm{w''}}.
\end{align*}
\end{proof}

Let 
\begin{align}
    E(r) = \bigcap_i F(i).
\end{align}

\begin{theorem}\label{thm: non-uniquely ergodic direction}
   $E(r)$ is a Cantor set in $\mathbb{RP}^1$ such that $E(r) \subset \Theta_\mathrm{NUE}(X,\omega)$.
\end{theorem}
\begin{proof}
We first show $E(r)$ is a Cantor set.
By construction, the set of children of each node is finite. By Theorem~\ref{th:induction step}, every node has at least two children, so $E(r)$ is nonempty and compact. 
The interval lengths at level $n$ are less than $\frac{1}{\norm{w^{(1)}}^{r^{n-1}}}$ which goes to $0$ when $n\rightarrow \infty$. 
It remains to show that $E(r)$ has no isolated points. Let $\theta\in E(r)$, and let $I_n$ be the unique level-$n$ interval containing $\theta$. Since every node has at least two
children, one may choose, at a sufficiently deep level, a descendant
interval different from the one containing $\theta$. Continuing along
this branch gives a point $\theta'\in E(r)\cap I_n$ with
$\theta'\neq\theta$. Since $\operatorname{diam}(I_n)\to0$, such
points can be chosen arbitrarily close to $\theta$. Thus $E(r)$ is
perfect. By Lemma~\ref{lem:subset}, the sets $F(n)$ form a nested sequence of nonempty compact sets. 
Moreover, Lemma~\ref{lem:separation} shows that the basic intervals at
each level are pairwise disjoint, while their diameters tend to zero.
Hence $E(r)$ is totally disconnected. So $E(r)$ is a Cantor set.

Next, we show that every $\theta\in E(r)$ is a non-uniquely ergodic direction.
We check condition $(1)$ in Theorem~\ref{thm:nonergodic direction}. We have
\begin{align*}
     \sum_{i=2}^\infty \norm{w^{(i)}\times v_1^{(i)}}<\sum_{i=2}^\infty \frac{1}{\log \norm{w^{(i)}}}\le \sum \frac{1}{r^{i-1}\log \norm{w^{(1)}}} <\infty.
\end{align*}
For an $H$-good twist,
$\norm{w^{(n)}\times v_1^{(n)}}\le \frac{1}{\log\|w^{(n)}\|}$,
and similarly for $v_2^{(n)}=\frac{\lambda}{\hat d}v_1^{(n)}$. Then
\begin{align*}
    \operatorname{Area}(C_1^{(n)})= d- \norm{w^{(n)}\times v_1^{(n)}}\geq d- \frac{1}{\log \norm{w^{(n)}}}\geq d- \frac{1}{\log W_*},
\end{align*}
and
\begin{align*}
    \operatorname{Area}(C_2^{(n)}) = \lambda^2 - \norm{w^{(n)}\times v_2^{(n)}}\geq \lambda^2- \frac{\lambda}{\log \norm{w^{(n)}}}\geq \lambda^2- \frac{\lambda}{\log W_*}.
\end{align*}
Then $\operatorname{Area}(C_i^{(n)})\ge \frac12\operatorname{Area}(T_i)$
uniformly in n. So condition (2) in Theorem~\ref{thm:nonergodic direction} is satisfied.

For $\theta\in E(r)$, we have $\theta \in I(w^{(n)})$. Then $$\dist_{\mathbb{RP}^1}(\theta,\theta(w^{(n)}))
\le
\frac1{2\|w^{(n)}\|^{r+1}}.$$ So $\dist_{\mathbb{RP}^1}(\theta,\theta(w^{(n)}))\rightarrow 0$ as $n\rightarrow \infty$. This implies that the directions of $w^{(n)}$ converge to $\theta$.
Hence the perpendicular component satisfies
$h_n\le C\|w^{(n)}\|\cdot \|w^{(n)}\|^{-(r+1)}
= C\|w^{(n)}\|^{-r}\to0$. This finishes the proof of condition (3) in Theorem~\ref{thm:nonergodic direction}. So $\theta$ is a nonergodic direction, which is also a non-uniquely ergodic direction. We have
\begin{align*}
    \theta\in \Theta_\mathrm{NUE}(X,\omega).
\end{align*}
\end{proof}

\begin{theorem}\label{thm:lower bound}
We have $\dim_H(E(r))\ge \frac{1}{r+1+(r-1)^2}$.
Then, for every $1<r<2$, and taking $r\downarrow1$ gives the $\dim_H \Theta_{\mathrm{NUE}}(X,\omega)$ to $1/2$. 
\end{theorem}
\begin{proof}
Let $m_n\ge2$ be the minimum number of level-$n+1$ intervals contained in a level-$n$ interval. If necessary, in each parent
interval we retain exactly \(m_n\) children. This gives a Cantor subset
\(E'(r)\subset E(r)\), and therefore
\[
\dim_H E(r)\ge \dim_H E'(r).
\]

Put \(L=\|w^{(1)}\|\). By Lemma \ref{lem:separation} and the upper bound in (\ref{eq:length}), any two
distinct retained intervals at level \(n+1\) are separated by at least
\[
\delta_n:=
\frac{1}{
2C_D^{\,2(r^n-1)/(r-1)}L^{2r^n}
}.
\]
The sequence $\left\{\delta_n\right\}$ is decreasing. Since all basic intervals
are contained in the initial projective interval \(I(w^{(1)})\), we may
identify this interval with a compact interval in \(\mathbb R\); this
identification is bi-Lipschitz and hence preserves Hausdorff dimension.
Thus Falconer's Cantor-set lower bound\cite{Fa} applies to \(E'(r)\), giving
\[
\dim_H E(r)\ge
\liminf_{n\to\infty}
\frac{\log(m_1\cdots m_{n-1})}
     {-\log(m_n\delta_n)}.
\]
By (\ref{eq:length}) and (\ref{eq: number}), $m_n\geq c_D\frac{(L^{r^{n-1})^{r-1}}}{r^{n-1}\log L + \frac{r^{n-1}-1}{r-1}\log C_D}$. So we have
\begin{align*}
\log(m_1\cdots m_{n-1})&\geq \sum_{j=1}^{n-1}r^{j-1}(r-1)\log L -\frac{n(n-1)}{2}\log r- O(n)\\
&\ge (r^{n-1}-1) \log L-\frac{n(n-1)}{2}\log r-O(n).
\end{align*}
Also,
$$-\log(m_n\delta_n)\le (r^n+r^{n-1})\log L+\frac{2r^n-1}{r-1}\log C_D +O(n).$$
Then the Hausdorff dimension of $E(r)$ satisfies:
\begin{align*}
    \dim_H(E(r)) 
             & \ge \liminf_{n \to \infty} \frac{(r^{n-1}-1) \log L-\frac{n(n-1)}{2}\log r-O(n)}{(r^n+r^{n-1})\log L+ (\frac{2r^n-1}{r-1}\log C_D)+O(n)}                    \\
             & = \frac{\log L}{(r+1)\log L +(\frac{2r}{r-1}\log C_D)}\\
             &= \frac{1}{r+1+ \frac{2r}{r-1}\frac{\log C_D}{\log L}}.
\end{align*}
By (\ref{def:W*}), $L>[1+6(1+\lambda)]^{\frac{2r}{(r-1)^3}}$ for $C_D=1+6(1+\lambda)$. Then for every $1<r<2$,
\begin{align*}
   \dim_H(E(r)) \ge \frac{1}{r+1+(r-1)^2}. 
\end{align*}
By Theorem~\ref{thm: non-uniquely ergodic direction}, 
$$\dim_H\Theta_{\mathrm{NUE}}(X,\omega)
\ge
\frac{1}{r+1+(r-1)^2}.$$ 
Letting $r\downarrow 1$, we obtain
$\dim_H \Theta_{\mathrm{NUE}}(X,\omega)\ge \frac{1}{2}$.
\end{proof}

\begin{proof}[Proof of \Cref{thm:main theorem}]
By Masur's theorem \cite{Ma2}, for every translation surface the set of non-uniquely ergodic
directions has Hausdorff dimension at most \(1/2\). This gives
\[
\dim_H \Theta_{\mathrm{NUE}}(X,\omega)\le 1/2.
\]
By \Cref{thm:lower bound},
\[
\dim_H \Theta_{\mathrm{NUE}}(X,\omega)\ge 1/2.
\]
\Cref{thm:main theorem} follows from the above two inequalities.
\end{proof}

\bibliography{ref}
\bibliographystyle{alpha}
\end{document}